\documentclass[11pt]{amsart}
\usepackage{amscd,amsmath,amssymb,epsfig,amsthm,latexsym}
\usepackage{graphicx}
\usepackage{tikz}
\usepackage{multirow}
\usepackage[enableskew,vcentermath]{youngtab}
\usepackage{hyperref}
\usepackage{tabu}
\hypersetup{colorlinks=false}

\newtheorem{theorem}{Theorem}[section]
\newtheorem{proposition}[theorem]{Proposition}
\newtheorem{corollary}[theorem]{Corollary}
\newtheorem{conjecture}[theorem]{Conjecture}

\newtheorem{example}[theorem]{Example}
\newtheorem{lemma}[theorem]{Lemma}
\newtheorem{remark}[theorem]{Remark}

\newtheorem{question}[theorem]{Question}

\newcommand{\diam}{{\rm diam}}

\newcommand{\out}{{\rm out}}

\newcommand{\aA}{{\mathcal A}}

\newcommand{\fF}{{\mathcal F}}
\newcommand{\hH}{{\mathcal H}}
\newcommand{\iI}{{\mathcal I}}

\newcommand{\pP}{{\mathcal P}}
\newcommand{\qQ}{{\mathcal Q}}
\newcommand{\rR}{{\mathcal R}}

\newcommand{\RR}{{\mathbb R}}

\renewcommand{\to}{\rightarrow}
\newcommand{\toto}{\longrightarrow}
\newcommand{\getto}{\longleftarrow}
 \newcommand{\sm}{{\smallsetminus}}
\begin{document}
\title[Enumeration of arborescences and monotone paths]
{Enumerative problems for Arborescences and \\ 
 Monotone Paths on Polytope Graphs}

\author[C.A.~Athanasiadis]{Christos~A.~Athanasiadis}
\address{Department of Mathematics\\
University of Athens\\
15784 Athens, Greece}
\email[C.A. Athanasiadis]{caath@math.uoa.gr}
\author[J.A.~De~Loera]{Jes\'us~A.~De~Loera}
\address{Department of Mathematics\\
University of California\\
Davis, CA 95616, USA}
\email[J.A. De~Loera]{deloera@math.ucdavis.edu}
\author[Z.~Zhang]{Zhenyang~Zhang}
\address{Department of Mathematics\\
University of California\\
Davis, CA 95616, USA}
\email[Z. Zhang]{zhenyangz@math.ucdavis.edu}

\date{\today}

\begin{abstract}
Every generic linear functional $f$ on a convex polytope $P$ induces
an orientation on the graph of $P$. From the resulting directed graph
one can define a notion of $f$-arborescence and $f$-monotone path on
$P$, as well as a natural graph structure on the vertex set of
$f$-monotone paths. These concepts are important in geometric
combinatorics and optimization.

This paper bounds the number of $f$-arborescences, the number of
$f$-monotone paths, and the diameter of the graph of $f$-monotone
paths for polytopes $P$ in terms of their dimension and number of
vertices or facets. 
\end{abstract}

\maketitle

\section{Introduction and results}
\label{sec:intro}

Consider a $d$-dimensional convex polytope $P$ in Euclidean 
space $\RR^d$ and a generic linear functional $f$ on $P$, 
meaning a linear functional on $\RR^d$ which is nonconstant 
on every edge of $P$. This paper investigates extremal 
enumerative problems about $f$-arborescences and 
$f$-monotone paths on the graph of $P$. We first introduce briefly these 
notions and refer to Section~\ref{sec:pre} for more 
information.

The functional $f$, which we think of as an objective 
function, induces an orientation on the graph of $P$ 
which orients every edge in the direction of increasing 
objective value. Such orientations of polytope graphs 
are called {\em LP-admissible}; they are of great importance 
in the study of the simplex method for linear 
optimization (see \cite{Develin2004, Klee+Mihalisin2000} 
and the references given there). The resulting directed 
graph, consisting of all vertices and oriented edges of 
$P$ and denoted by $\omega(P,f)$, is acyclic and has a 
unique source and a unique sink on every face of $P$. 
\begin{figure}
\begin{center}
  \includegraphics[scale = 0.25]{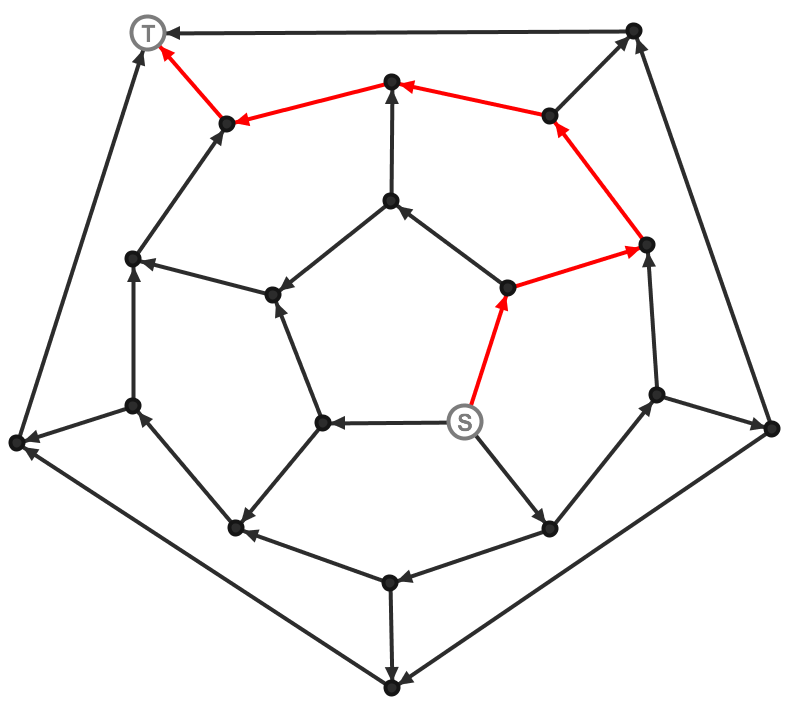}  \includegraphics[width = 0.25\linewidth]{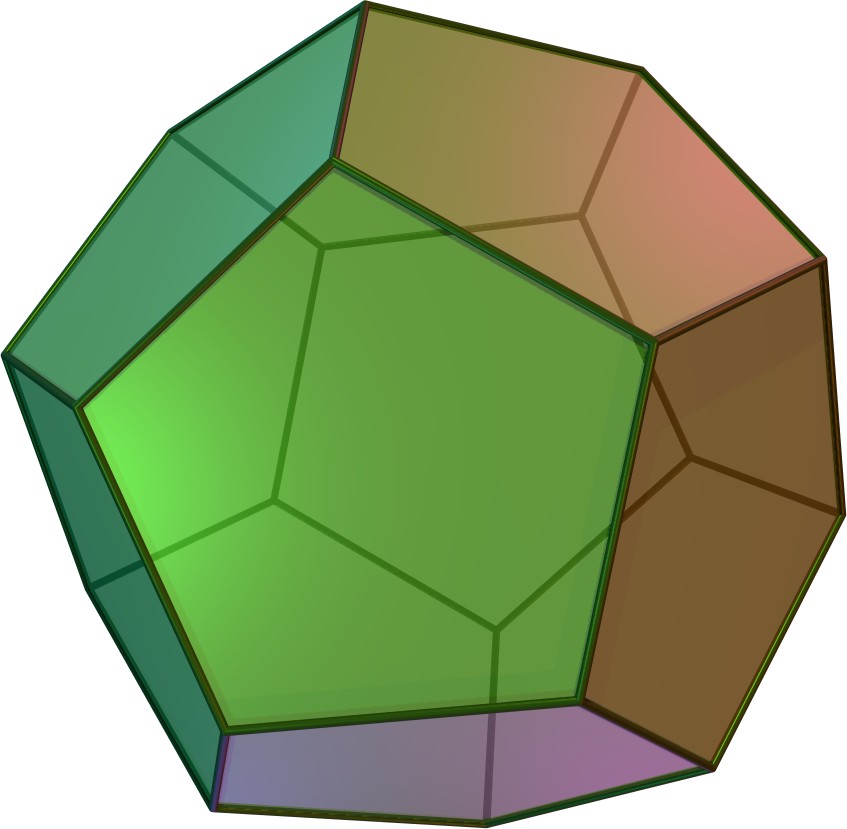}
\includegraphics[width = 0.34\linewidth]{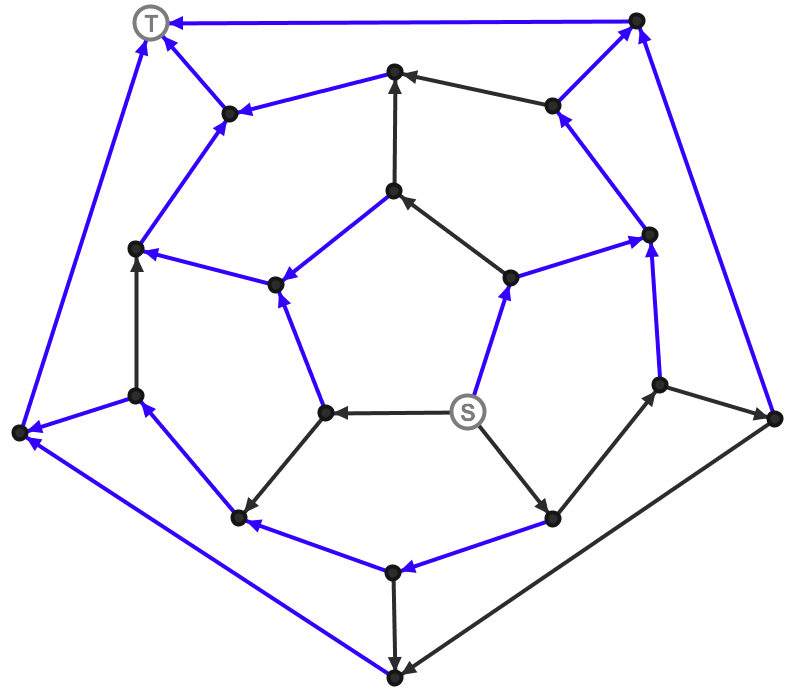}  
\end{center}  
\caption{The regular dodecahedron (center), with examples of an $f$-monotone path (left) and an $f$-arborescence (right) for one of its LP-admissible orientations.} \label{dodecaoriented}
\end{figure}
An \emph{$f$-monotone path} on $P$ is any directed path 
in $\omega(P,f)$ having as initial and terminal vertex 
the unique source, say $v_{\min}$, and the unique sink, 
say $v_{\max}$, of $\omega(P,f)$ on $P$, respectively. 
An \emph{$f$-arborescence} is any (necessarily acyclic) 
spanning subgraph $\aA$ of the directed graph 
$\omega(P,f)$ such that for every vertex $v$ of $P$ 
there exists a unique directed path in $\aA$ with 
initial vertex $v$ and terminal vertex $v_{\max}$ (see Figure \ref{dodecaoriented} for an example).
As explained in the sequel, $f$-arborescences and 
$f$-monotone paths are important notions in geometric 
combinatorics and optimization. When the context is 
clear, we simply refer to them as arborescences and 
monotone paths. 

The set of all $f$-monotone paths on $P$ can be given 
a natural graph structure as follows. We say that two 
$f$-monotone paths on $P$ differ by a \emph{polygon flip} 
(also called \emph{polygon move}, or simply \emph{flip}) 
across a $2$-dimensional face $F$ if they agree 
on all edges not lying on $F$ but follow the two different 
$f$-monotone paths on $F$, from the unique source to the 
unique sink of $\omega(P,f)$ on $F$. The \emph{graph of 
$f$-monotone paths} (also called \emph{flip graph}) 
on $P$ is denoted by $G(P,f)$ and is defined as the simple 
(undirected) graph which has nodes all $f$-monotone 
paths on $P$ and as edges all unordered pairs of such 
paths which differ by a polygon flip across a 
$2$-dimensional face of $P$. 
The graph $G(P,f)$ is  connected; its higher connectivity was studied in 
\cite{AER00}, where it was shown that $G(P,f)$ is 
2-connected for every polytope $P$ of dimension $d \ge 
3$ and $(d-1)$-connected for every simple polytope 
$P$ of dimension $d$ (see Figure \ref{flipgraphexample} for 
an example).

\begin{figure}
\begin{center}
\includegraphics[scale = .15]{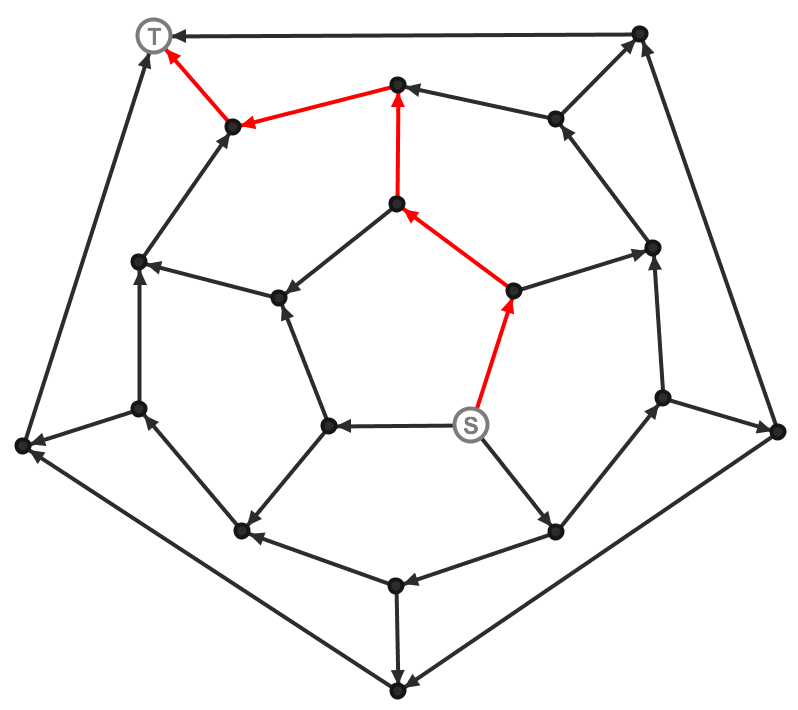}  \includegraphics[scale = 0.15]{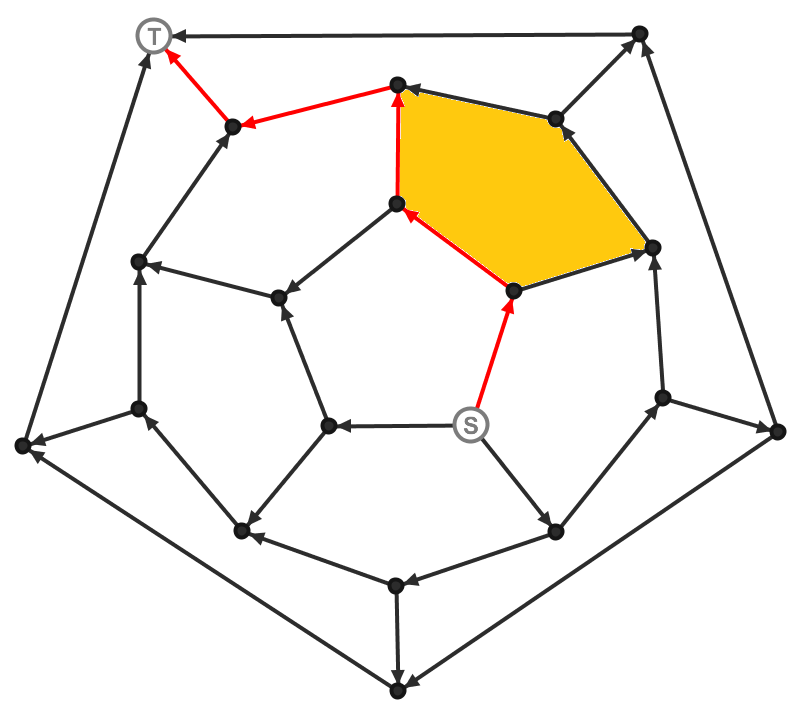}   \hskip .2cm \includegraphics[scale = 0.15]{images/dodeca-1.png}  
\vskip .1cm
\noindent \includegraphics[scale = 0.09]{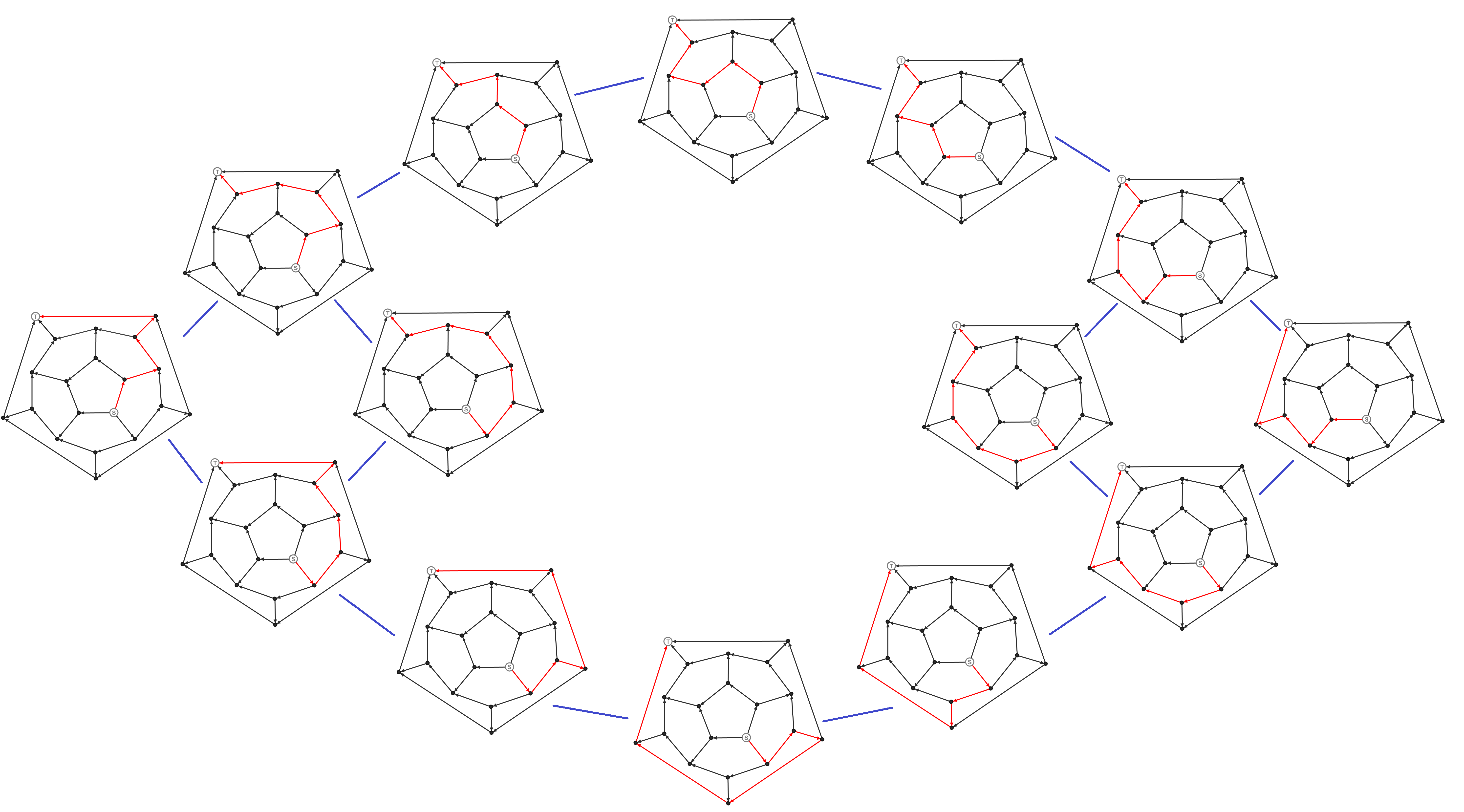}  
\end{center}
\caption{A polygon flip on the oriented dodecahedron and the resulting flip graph} \label{flipgraphexample}
\end{figure}

The main questions addressed in this paper ask to determine:
\begin{itemize}
\item 
the minimum and maximum number of $f$-arborescences 
on $P$,

\item 
the minimum and maximum number of $f$-monotone paths 
on $P$, and

\item 
the minimum and maximum diameter of the graph $G(P,f)$,
\end{itemize}
where $P$ ranges over all convex polytopes of given 
dimension and number of vertices and $f$ ranges over 
all generic linear functionals on $P$. We will also 
consider these (or similar) questions when $P$ is 
restricted to the important class of simple polytopes.

There are good reasons, from both a theoretical and 
an applied perspective, to study these problems. One 
motivation comes from the connection of $f$-arborescences
and $f$-monotone paths to the behavior of the 
simplex method \cite{Schrijver1986}. The simplex 
method produces a partial $f$-monotone path, 
traversing $\omega(P,f)$ from an initial vertex to 
the optimal one. The simplex method has to make 
decisions to choose the improving arcs via a 
\emph{pivot rule}. It is an open problem to find the 
longest possible simplex method paths and little is 
known about bounds (see \cite{BDL2019} and references 
therein). Clearly, the lengths of $f$-monotone paths 
are of great interest, as they bound the number of 
steps in the simplex algorithm. A pivot rule gives a 
mapping from the set of instances of the algorithm to 
the set of $f$-arborescences of $\omega(P,f)$. Two pivot 
rules are equivalent if they always produce the same 
$f$-arborescence. Therefore, given $P$ and $f$, there 
are only finitely many equivalence classes of pivot 
rules and counting $f$-arborescences is a proxy for 
the problem of counting pivot rules. 

Another motivation comes from enumerative and 
polyhedral combinatorics, especially from the theory 
of fiber polytopes \cite{BS92}. The flip graph of 
$f$-monotone paths on $P$ contains a well behaved 
subgraph, namely that induced on the set of 
\emph{coherent} $f$-monotone paths (these are the 
monotone paths which come from the shadow vertex pivot rule 
\cite{DH16}). This subgraph is isomorphic to the 
graph of a convex polytope of dimension $d-1$, where 
$d = \dim(P)$, which is a fiber polytope known as 
a \textit{monotone path polytope} \cite[Section~5]{BS92} 
\cite{BKS94}. 
Monotone paths, monotone path polytopes 
and flip graphs of polytopes of combinatorial interest 
often have elegant combinatorial interpretations. For 
example, the monotone path polytope of a cube is a 
permutohedron \cite[Example~5.4]{BS92}, while the flip 
graph of the latter encodes reduced decompositions of 
a certain permutation and the braid relations among 
them \cite[Section~2.4]{BLSWZ}. More generally, monotone 
paths on 
zonotopes \cite{AtSa2001,RR13} correspond to certain 
galleries of chambers in a central hyperplane 
arrangement and the problem to estimate the diameter 
of the flip graph in this important special case has 
been intensely studied in \cite{Edman-thesis, edman2018zonotopes, RR13}. 
The diameter of flip graphs of fiber polytopes has 
also been studied in \cite{Pournin_2014, Pournin_2017}. 
Moreover, certain zonotopes are in fact monotone path 
polytopes coming from projecting cyclic polytopes 
\cite[Section~3]{ADRS00}, or polytopes which look
like piles of cubes~\cite{At99}. Monotone path 
polytopes are also related to fractional power series 
solutions of algebraic equations \cite{McDonald}. 
The combinatorial properties of $f$-monotone paths 
and flip graphs have thus been studied in comparison to 
those of coherent $f$-monotone paths, but also because 
of their own independent interest.

A special role in our results is played by a 
distinguished member $X(n)$ of the family of stacked 
3-dimensional simplicial polytopes with $n$ vertices. 
As it turns out, this polytope maximizes the number 
of both $f$-arborescences and $f$-monotone paths, and 
possibly the diameter of the flip graph too, in this 
dimension. We refer to Section~\ref{sec:stack} for 
a discussion of stacked polytopes and the precise 
definition of $X(n)$, which we always 
consider endowed with the specific LP-allowable 
orientation given there. We will typically denote by 
$n$ (and sometimes by $n+1$) and $m$ the number of vertices 
and facets of $P$, respectively. Let us also denote by 
\begin{itemize}
\item 
$\tau(P,f)$ the number of $f$-arborescences on $P$,

\item 
$\mu(P,f)$ the number of $f$-monotone paths on $P$,

\item 
$\diam(G)$ the diameter of the graph $G = G(P, f)$.
\end{itemize}

Our first two main results provide a fairly complete 
description of tight bounds for the numbers of 
$f$-arborescences and $f$-monotone paths and the 
diameter of the graph of $f$-monotone paths on a 
$3$-dimensional polytope with given number of vertices. 
The upper bound for the number of $f$-monotone paths 
involves the sequence of 
Tribonacci numbers (sequence A000073 in~\cite{Sloane}), 
defined by the recurrence $T_0 = T_1 = 1$, $T_2 = 2$ 
and $T_n = T_{n-1} + T_{n-2} + T_{n-3}$ for $n \ge 3$. 
\begin{theorem} \label{thm:main-3d}
For $n \ge 4$, 
\begin{align} 
2(n-1) & \ \le \ \tau(P, f) \ \le \ 2 \cdot 3^{n-3} 
\label{eq:tau-bounds-3d} \\ 
\left\lceil {\frac{n}{2}} \right\rceil + 2 & \ \le \ 
\mu(P, f) \ \le \ T_{n-1}
\label{eq:mu-bounds-3d}
\end{align}
for every 3-dimensional polytope $P$ with $n$ vertices 
and every generic linear functional $f$ on $P$. 
The upper bound is achieved by the stacked polytope 
$X(n)$ in both situations. 

The lower bound of 
(\ref{eq:tau-bounds-3d}) can be achieved by pyramids 
and that of (\ref{eq:mu-bounds-3d}) by prisms, when $n$ 
is even, and by wedges of polygons over a vertex, when 
$n$ is odd. In particular, prisms minimize the number 
of $f$-monotone paths over all simple 3-dimensional 
polytopes with given number of vertices. Moreover,
\[ \tau(P,f) \, = \, 3 \cdot 2^{(n-2)/2} \, = \, 
   3 \cdot 2^{m-3} \]
for every 3-dimensional simple polytope $P$ with $n$ 
vertices and $m$ facets. 
\end{theorem} 

\begin{theorem} \label{thm:diam-max}
For every $n \ge 4$, 
\begin{equation}
\label{eq:diam-max}
\lceil \frac{(n-2)^2}{4} \rceil \le \max \diam \, 
G(P,f) \le (n-2) \lfloor \frac{n-1}{2} \rfloor, 
\end{equation}
where $P$ ranges over all 3-dimensional polytopes 
with $n$ vertices and $f$ ranges over all generic 
linear functionals on $P$.
\end{theorem}

Our results are substantially weaker in dimensions
$d \ge 4$ and leave plenty of room for further research.
The upper bounds for the number of $f$-arborescences and
the number of $f$-monotone paths are almost trivial, but
are included here for the sake of completeness.
\begin{theorem} \label{thm:main-d>3}
\begin{enumerate}
\item[(a)]
For $n > d \ge 4$,
\begin{align*}
\tau(P, f) & \le (n-1)! \\
\mu(P, f) & \le 2^{n-2}
\end{align*}
for every $d$-dimensional polytope $P$ with $n$
vertices and every generic linear functional $f$ on $P$.
These bounds are achieved by any 2-neighborly $d$-dimensional
polytope with $n$ vertices.

\item[(b)]
For $m > d \ge 4$,
\[ d \cdot \left( (d-1)! \right)^{m-d} \le \tau(P,f)
  \le \prod\limits_{i=1}^{\lfloor \frac{d}{2} \rfloor} \,
	 i^{{m-d+i-1 \choose i}}
        \prod\limits_{i = \lfloor \frac{d+1}{2} \rfloor}^d \,
	 i^{{m-i-1 \choose d-i}} \]
for every simple $d$-dimensional polytope $P$ with $m$
facets and every generic linear functional $f$ on $P$.
The lower and upper bounds are achieved by the polar
duals of stacked simplicial polytopes and the polar
duals of neighborly simplicial polytopes, respectively,
of dimension $d$ with $m$ vertices.
\end{enumerate}
\end{theorem}

%

The proofs of the results on $f$-arborescences, given 
in Section~\ref{sec:arbo}, rely on the fact that 
$\tau(P,f)$ is equal to the product 
of the outdegrees of the vertices of the directed 
graph $\omega(P, f)$ other than the sink (see 
Proposition~\ref{prop:product}). This has the 
curious consequence that $\tau(P, f)$ is independent 
of $f$ for every simple polytope $P$. 
The proofs of the results on $f$-monotone paths and 
the diameter of flip graphs, given in 
Sections~\ref{sec:paths} and~\ref{sec:diameter},
respectively, use ideas from \cite[Section~4]{AER00} 
\cite{BKS94}, reviewed in Section~\ref{sec:graph}, 
to construct $G(P,f)$ as an inverse limit in the 
category of graphs and simplicial maps. 
Section~\ref{sec:pre} contains preliminary 
material on polytopes, needed to understand the 
main results and their proofs, defines the 
stacked polytope $X(n)$ and proves a 
combinatorial lemma about its diameter 
(Lemma~\ref{lem:diam-stack}) which implies the 
lower bound in Theorem~\ref{thm:diam-max}. 
Section~\ref{sec:conclusions} concludes with 
comments about the missing bounds and related 
open problems.

\section{Preliminaries}
\label{sec:pre}

This section reviews basic background and terminology on 
convex polytopes and monotone paths and discusses a few 
constructions and a preliminary result 
(Lemma~\ref{lem:diam-stack}) which will be useful in the 
sequel. We use the notation $[n] := \{1, 2,\dots,n\}$ 
for any positive integer $n$ and refer the reader to the 
book~\cite{Ziegler} for any undefined concepts and 
terminology. In particular, the last chapter contains an introduction to fiber polytopes.

\subsection{Some special polytopes}
\label{sec:stack}

Special classes of polytopes play an important role in this 
paper, since they are optimal solutions of the extremal 
problems considered. Recall that a polytope is called 
\emph{simplicial} if all its proper faces are simplices. 
The \emph{simple polytopes} are the polar duals of 
simplicial polytopes. A convenient way to encode the 
numbers of faces of each dimension of a simple or 
simplicial $d$-dimensional polytope $P$ is provided by 
the \emph{$h$-vector}, denoted as $h(P) = (h_0(P), 
h_1(P),\dots,h_d(P))$; see pages 8, 59 and 248 of 
\cite{Ziegler} for details and more information. The 
$h$-vector of a simple polytope $P$ has nonnegative 
integer coordinates which afford an 
elegant combinatorial interpretation: $h_k(P)$ equals the 
number of vertices of $P$ of outdegree $k$ in the directed 
graph $\omega(P,f)$, discussed in the introduction, for 
every generic linear functional $f$ on $P$ (see Sections 
3.4 and 8.3 and Exercise 8.10 in \cite{Ziegler}); in 
particular, the multiset of such outdegrees is 
independent of $f$. 

A polytope is called \emph{2-neighborly} if every pair of 
vertices is connected by an edge. A $d$-dimensional  
simplicial polytope is called \emph{neighborly} if any 
$\lfloor d/2 \rfloor$ or fewer of its vertices form the 
vertex set of a face. Neighborly polytopes other than 
simplices (cyclic polytopes being distinguished 
representatives) exist in dimensions four and higher. 
Their significance comes from the fact that they maximize 
the entries of the $h$-vector among all polytopes with 
given dimension and number of vertices (see pages 15-16, 
and 254-257 of \cite{Ziegler}); in particular, 
they maximize the numbers of faces of each dimension 
among such polytopes. The $h$-vector of a neighborly
$d$-dimensional polytope $P$ with $n$ vertices is given
by the formulas $h_k(P) = {n-d+k-1 \choose k}$ for $0 \le
k \le \lfloor d/2 \rfloor$ and $h_k(P) = h_{d-k}(P)$ for
$0 \le k \le d$ (see Theorem 8.21 and Lemma 8.26 of \cite{Ziegler}).

\begin{figure}[h]
\centering
\includegraphics[scale = 0.3]{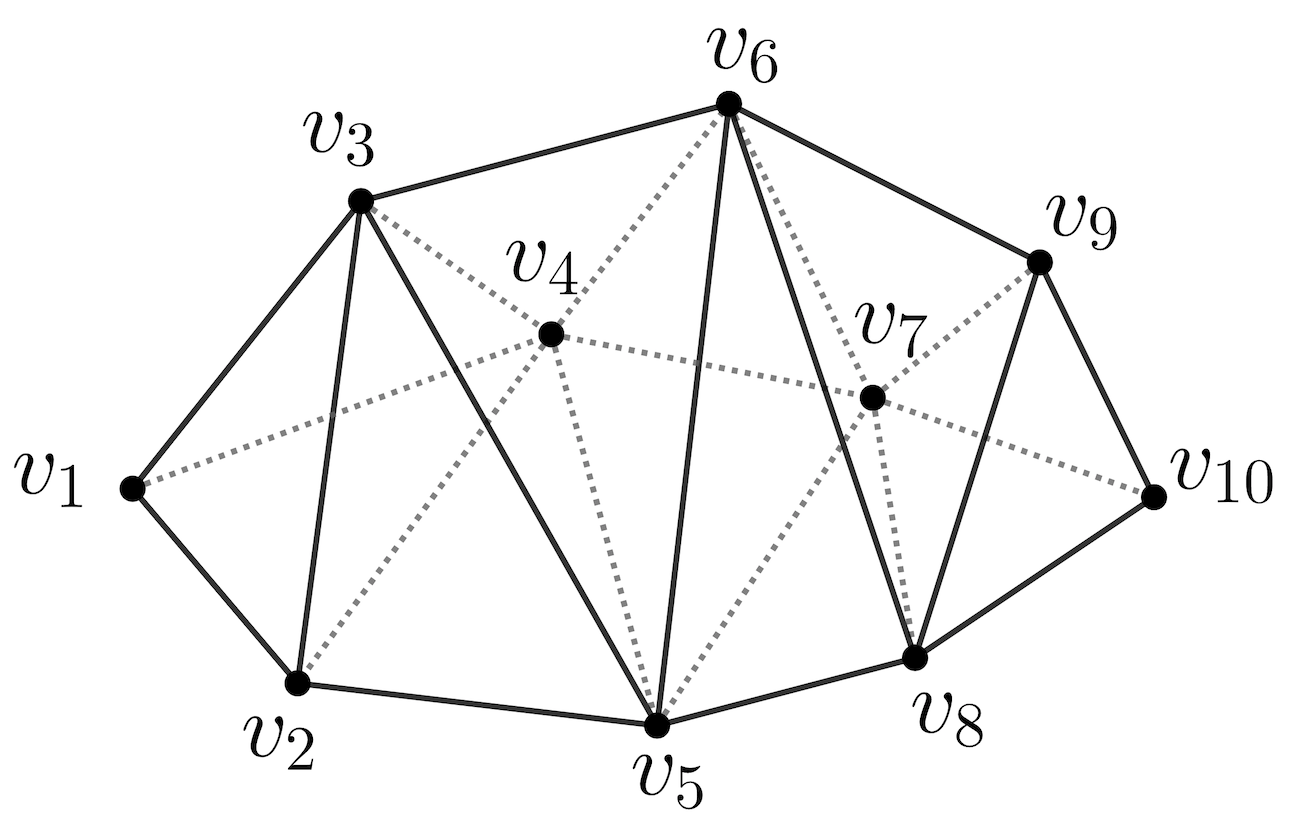}  
\caption{Example of the polytope $X(10)$} \label{Xtacked}
\end{figure}

A \emph{stacked polytope} is any simplicial polytope 
which can be obtained from a simplex by repeatedly 
glueing other simplices of the same dimension along 
common facets, so as to preserve convexity at each 
step. Equivalently, the boundary complex of a stacked 
polytope can be obtained combinatorially from that of 
a simplex by successive stellar subdivisions on 
facets. The $h$-vector of any stacked polytope $P$ 
of dimension $d$ with $n$ vertices has the simple form 
$h(P) = (1,n-d,...,n-d,1)$ (see \cite{McMullen-triangpaper}). 
A fundamental result of Barnette~\cite{barnetteLBT} 
states that among all simplicial polytopes with given 
dimension and number of vertices, the stacked polytopes 
have the fewest possible faces of each dimension. 
Moreover, as a consequence of the generalized lower 
bound theorem, stacked polytopes minimize the entries 
of the $h$-vector among all such polytopes 
(see \cite{KalaiLBT, Murai+Nevo}). 

Many different combinatorial types of stacked polytopes 
are possible. For each $n \ge 4$, we will consider a 
3-dimensional stacked polytope of special type with $n$ 
vertices, denoted by $X(n)$. This polytope comes together 
with a linear functional $f$ which linearly orders its 
vertices as $f(v_1) < f(v_2) < \cdots < f(v_n)$. The 
associated triangulation comprises of all faces of the 
simplices with vertex 
sets $\{v_1, v_2, v_3, v_4\}$, $\{v_2, v_3, v_4, 
v_5\},\dots,\{v_{n-3}, v_{n-2}, v_{n-1}, v_n\}$, so
the dual graph of this triangulation is a path (these 
dual graphs for general stacked polytopes are trees). 
The regularity of this triangulation
easily implies that such polytope $X(n)$ and 
linear functional $f$ exist for every $n \ge 4$. 
Figure~\ref{Xtacked} shows an example with $n=10$. 

A crucial property of $X(n)$ is that the directed graph
$\omega(X(n),f)$ has as arcs the pairs $(v_i, v_j)$ for 
$i, j \in \{1, 2,\dots,n\}$ with $j \in \{i+1, i+2, i+3
\}$. The following combinatorial lemma establishes the 
lower bound for the diameter of flip graphs, claimed in 
Theorem~\ref{thm:diam-max}.

\begin{lemma} \label{lem:diam-stack}
The diameter of the graph of $f$-monotone paths on $X(n)$ 
is bounded below by $\lceil (n-2)^2/4 \rceil$ for every 
$n \ge 4$.
\end{lemma}
\begin{proof}
Let $G$ be the graph of $f$-monotone paths on $X(n)$. 
Denoting $f$-monotone paths as sequences of vertices, 
we set
\[ \gamma \, = \, \begin{cases}
   (v_1, v_3, v_5,\dots,v_{n-1}, v_n), & 
	          \text{if $n \equiv 0 \; (\bmod{\, 2})$} \\
   (v_1, v_2, v_4,\dots,v_{n-3}, v_{n-1}, v_n), & 
	          \text{if $n \equiv 1 \; (\bmod{\, 4})$} \\
	 (v_1, v_3, v_5,\dots,v_{n-2}, v_n), & 
	          \text{if $n \equiv 3 \; (\bmod{\, 4})$}
\end{cases} \]
and $\delta = (v_1, v_2, v_3,\dots,v_n)$. We claim 
that $\gamma$ and $\delta$ are at a distance 
of $\lceil (n-2)^2/4 \rceil$ apart in $G$. Clearly, 
the lemma follows from the claim.

We only consider the case that $n$ is even, the other
two cases being similar. By passing to the complement 
of the set of vertices appearing on an 
$f$-monotone path on $X(n)$, such paths correspond 
bijectively to the subsets of $\{v_2, v_3,\dots,v_{n-1}\}$ 
containing no three consecutive elements $v_{k-1}, v_k, 
v_{k+1}$. The subset which corresponds to $\gamma$, for 
instance, is $\{v_2, v_4,\dots,v_{n-2}\}$ and the one 
which corresponds to $\delta$ is the empty set. The 
2-dimensional faces of $X(n)$ have vertex sets 
$\{v_1, v_2, v_3\}$, $\{v_{n-2}, v_{n-1}, v_n\}$ and
$\{v_{k-1}, v_k, v_{k+2}\}$ and $\{v_{k-1}, v_{k+1}, 
v_{k+2}\}$ for $2 \le k \le n-2$. From these 
facts it follows that polygon flips on $f$-monotone
paths on $X(n)$ correspond to the following operations
on the corresponding subsets:

\begin{itemize}
\item[$\bullet$]
removal of $v_2$ or $v_{n-1}$, if present,
\item[$\bullet$]
inclusion of $v_2$, if absent and not both $v_3$ and $v_4$ 
are present,
\item[$\bullet$]
inclusion of $v_{n-1}$, if absent and not both $v_{n-2}$ and 
$v_{n-3}$ are present,
\item[$\bullet$]
removal or inclusion of one of $v_k, v_{k+1}$, if the other 
is present but $v_{k-1}$ and $v_{k+2}$ are absent.
\end{itemize}

Since the subsets which correspond to $f$-monotone paths on
$X(n)$ contain no three consecutive elements, their maximal
strings of consecutive elements are either singletons, or
contain exactly two elements. Moreover, the strings
cannot be merged with these operations, they cannot be removed
except for $\{2\}$ and $\{n-1\}$, and each operation affects
only one of them. To reach the empty
set from $\{v_2, v_4,\dots,v_{n-2}\}$, one needs to remove
each of $v_2, v_4,\dots,v_{n-2}$. Regardless of the order in
which operations are performed, at least one is needed to
remove $v_2$, at least three more are needed to remove
$v_{n-2}$, at least five more are needed to remove $v_4$, and
so on. For example, to remove $v_{n-2}$ in at most three steps
one needs to first include $v_{n-1}$, then remove $v_{n-2}$
and finally remove $v_{n-1}$ and to remove $v_4$ in at most
five steps one needs to first include $v_3$, then remove $v_4$,
include $v_2$, remove $v_3$ and finally remove $v_2$. This
yields a distance of $1 + 3 + 5 + \cdots + (n-3) = (n-2)^2/4$
between $\gamma$ and $\delta$ in $G$.
\end{proof}

\begin{remark} \label{rem:paths+flips}
\rm Perhaps it is instructive to visualize the process of 
flipping $\gamma$ to $\delta$, described in the previous 
proof. The two $f$-monotone paths are shown on 
Figure~\ref{colored-paths} for $n=10$ and the sequence of 
2-dimensional faces (recording only vertex indices, for 
simplicity) across 
which the flips occur could be $\{1, 2, 3\}$, $\{7, 9, 10\}$, 
$\{7, 8, 10\}$, $\{8, 9, 10\}$, $\{2, 3, 5\}$, 
$\{2, 4, 5\}$, $\{1, 2, 4\}$, $\{1, 3, 4\}$, $\{1, 2, 3\}$, 
$\{5, 7, 8\}$, $\{5, 6, 8\}$, $\{6, 8, 9\}$, $\{6, 7, 9\}$,
$\{7, 9, 10\}$, $\{7, 8, 10\}$ and $\{8, 9, 10\}$.
\qed
\end{remark}

\begin{figure}[h]
\centering
    \includegraphics[scale = 0.3]{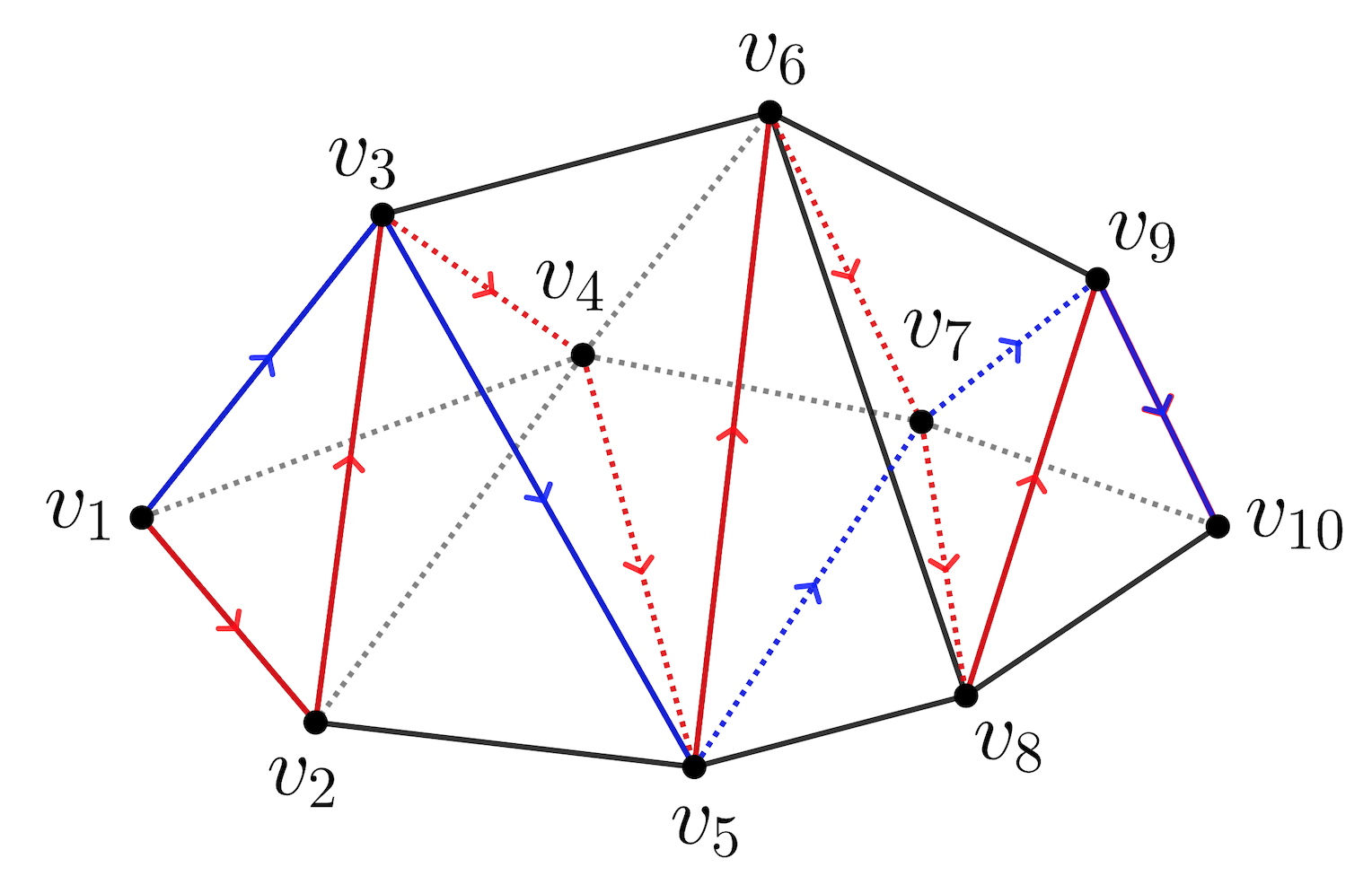}
    \caption{Two monotone paths on $X(10)$} \label{colored-paths}
\end{figure}

Finally, we consider prisms and wedges of polygons. Given 
a $(d-1)$-dimensional polytope $Q$, the \emph{prism} over 
$Q$ is the $d$-dimensional polytope defined as the Cartesian 
product $Q \times [0,1]$. The \emph{wedge} of $Q$ over a 
face $F$ of $Q$ is the $d$-dimensional polytope $W$ 
obtained combinatorially from the prism $Q \times [0,1]$ 
by collapsing the face $F \times [0,1]$ to $F \times {0}$. 
Note that $Q$ becomes a facet of $W$ and that if $F$ is a 
facet and $Q$ is simple, then so is $W$. We will apply the 
wedge construction in the special cases that $Q$ is a 
polygon and $F$ is one of its vertices or edges. 

\begin{figure}[h]
\begin{center}
\includegraphics[scale = 0.6]{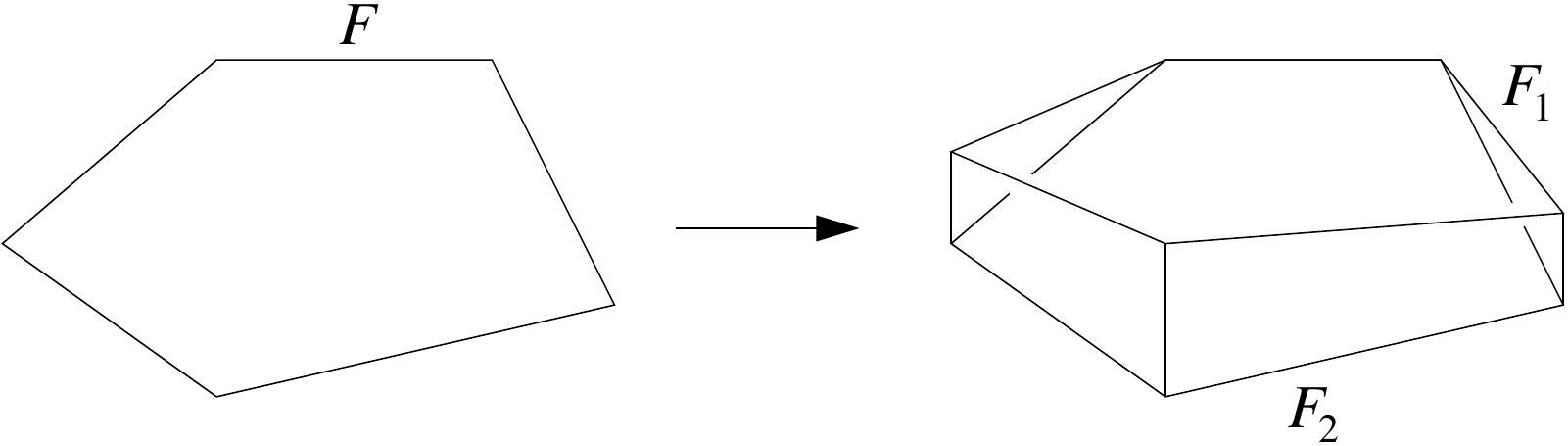}
\end{center}
\caption{The wedge of a pentagon over an edge}
\end{figure}

\subsection{The graph of $f$-monotone paths}
\label{sec:graph}

Let $P$ be a $d$-dimensional polytope and $f$ be a generic 
linear functional on $P$. We will assume that $f$ does not 
take the same value on any two distinct vertices of $P$. 

To investigate the graph of $f$-monotone paths on $P$, we 
will describe another way to construct it from simpler graphs, 
arising in the fibers of the restriction of the projection 
map $f$ on $P$. 
The technical device needed, which we now review, is the 
\emph{inverse limit} in the category of graphs and simplicial 
maps. This concept was introduced in \cite[Section~4]{AER00} 
(with motivation coming from \cite{BKS94}) to study the higher 
connectivity of $G(P,f)$; it leads to various more general 
graphs of partial $f$-monotone paths on $P$, a useful notion 
which allows for inductive arguments.

Let us linearly order the vertices $v_0, v_1,\dots,v_n$ of 
$P$ so that $f(v_0) < f(v_1) < \cdots < f(v_n)$. We recall that 
for every interior point $t$ of the interval $f(P)$, the fiber 
$P(t) := f^{-1}(t) \cap P$ of the map $f: P \to \RR$ is a
$(d-1)$-dimensional polytope and thus it has a well defined 
graph. Setting $t_i = f(v_i)$ for $0 \le i \le n$, we may 
thus consider the graph $G_i$ of $P(t_i)$ for $0 \le i \le n$ 
and the graph $G_{i,i+1}$ of $P(t)$ for some $t_i < t < 
t_{i+1}$, for $0 \le i \le n-1$ (the precise value of $t$ 
being irrelevant because, by construction, the other choices 
of $t$ in the same interval give a normally equivalent fiber 
$P(t)$); see Figure~\ref{cube-fibers} for an example. 
Considering these graphs as one-dimensional simplicial 
complexes, we have a diagram  

\begin{equation}
\label{eq:main-diagram}
G_{0,1} \ \overset{\alpha_1}{\toto} \ G_1 \ 
\overset{\beta_1}{\getto} \ G_{1,2} \ \overset{\alpha_2}{\toto} 
\ G_2 \ \overset{\beta_2}{\getto} \ \cdots \ 
\overset{\beta_{n-2}}{\getto} \ G_{n-2,n-1} \   
\overset{\alpha_{n-1}}{\toto} \ G_{n-1} \ \overset{\beta_{n-1}}
{\getto} \ G_{n-1,n}
\end{equation}

\bigskip
\noindent
of graphs and simplicial maps for which $\alpha_i: G_{i-1,i} 
\to G_i$ and $\beta_i: G_{i,i+1} \to G_i$ result from the 
degeneration of the fiber $P(t)$ when $t$ approaches $t_i$, 
with $t_{i-1} < t < t_i$ or $t_i < t < t_{i+1}$, respectively 
(recall that a simplicial map of one-dimensional simplicial 
complexes maps vertices to vertices and either maps edges 
linearly onto edges, or contracts them to vertices; in 
particular, such a map is determined by its images on 
vertices). 

The inverse limit $G$ of this diagram is defined as follows. 
The nodes are the sequences
\[ (v_{0,1}, v_{1,2},\dots,v_{n-1,n}), \]
where $v_{i-1,i}$ is a vertex of $G_{i-1,i}$ for all $i \in 
[n]$ and $\alpha_i(v_{i-1,i}) = \beta_i(v_{i,i+1})$ for all 
$i \in [n-1]$. Two such sequences, say 
$(u_{0,1}, u_{1,2},\dots,u_{n-1,n})$ and 
$(v_{0,1}, v_{1,2},\dots,v_{n-1,n})$, are adjacent nodes in $G$ 
if there exists a nonempty interval $\iI \subseteq [n]$ such 
that: 

\medskip
\begin{itemize}
\item[$\bullet$]
$u_{i-1,i}$ and $v_{i-1,i}$ are adjacent in $G_{i-1, i}$ for 
$i \in \iI$,
\item[$\bullet$] 
$u_{i-1,i} = v_{i-1,i}$ for $i \in [n] \sm \iI$, and
\item[$\bullet$]
the edges $\{u_{i-1,i}, v_{i-1,i}\}$ and $\{u_{i,i+1}, 
v_{i,i+1}\}$ are mapped homeomorphically onto the same edge of 
$G_i$ by $\alpha_i$ and $\beta_i$, respectively, whenever 
$i, i+1 \in \iI$.
\end{itemize}

This construction associates an inverse limit graph to any 
diagram of graphs and simplicial maps (\ref{eq:main-diagram}).
As explained in \cite[Section~4]{AER00} (see 
\cite[Proposition~4.1]{AER00}), the graph $G$ is isomorphic 
to $G(P,f)$ when the diagram comes from a polytope $P$ and 
linear functional $f$, as just described. The inverse limit 
of a subdiagram of~(\ref{eq:main-diagram}) of the form
\[ G_{k-1,k} \ \overset{\alpha_k}{\toto} \ G_k \ 
\overset{\beta_k}{\getto} \ G_{k,k+1} \ 
\overset{\alpha_{k+1}}{\toto} \ \cdots \ 
\overset{\beta_{\ell-1}}{\getto} \ G_{\ell-1,\ell} \ 
\overset{\alpha_\ell}{\toto} \ G_\ell \ \overset{\beta_{\ell}}
{\getto} \ G_{\ell,\ell+1} , \]
considered in Sections~\ref{sec:paths} and~\ref{sec:diameter}, 
has nodes which can be viewed as partial $f$-monotone paths 
on $P$, starting at the fiber $P(t)$ with $t_{k-1} < t < t_k$ 
and ending at $P(t')$ with $t_\ell < t' < t_{\ell+1}$, and 
adjacency given by a suitable extension of the notion of 
polygon flip, presented in the introduction.


\begin{figure}[h]
\includegraphics[scale = 0.65]{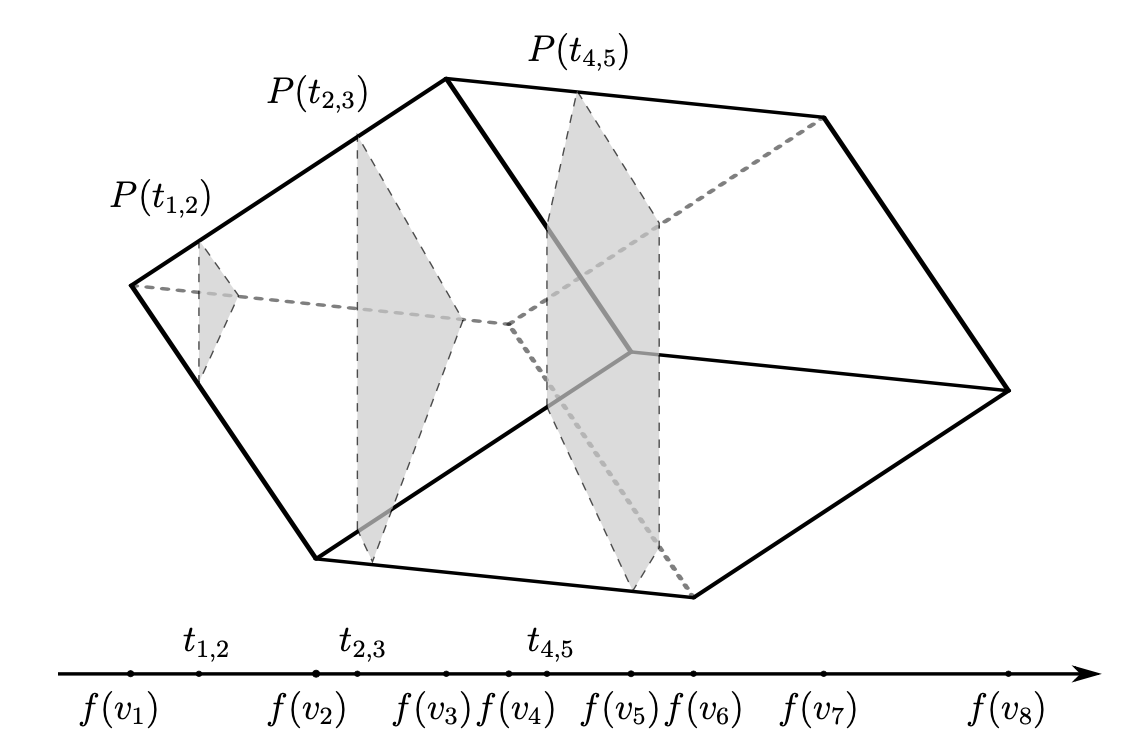}
\caption{A combinatorial cube and some of its fibers}
\label{cube-fibers}
\end{figure}

\section{On the number of arborescences}
\label{sec:arbo}

As explained in the introduction, we are interested in 
counting $f$-arborescences on a polytope $P$, meaning 
oriented spanning trees in the directed graph $\omega(P,f)$ which 
are rooted at the unique sink $v_{\max}$. Recall that 
$\tau(P,f)$ denotes the number of $f$-arborescences on 
$P$. The following statement provides an explicit 
product formula for this number.
\begin{proposition} \label{prop:product}
Given a $d$-dimensional polytope $P$ and generic linear 
functional $f$, let $\out_f(v)$ denote the outdegree of 
the vertex $v$ of $P$ in the directed graph $\omega(P, f)$. 
Then, 
\[ \tau(P, f) \, = \prod\limits_{v \ne v_{\max}} 
    \out_f(v), \]
where the product ranges over all vertices of $P$ other than the sink $v_{\max}$. In particular, 
if P is simple, then 
\[ \tau(P, f) \, = \, \prod\limits_{i=1}^d i^{h_i(P)} \]
is independent of $f$.
\end{proposition}

\begin{proof}
Since $\omega(P, f)$ is acyclic, an $f$-arborescence is
uniquely determined by a choice of edge coming out of $v$
for every vertex $v$ of $\omega(P, f)$ other than the sink
$v_{\max}$. Since there are exactly $\out_f(v)$ choices
for every such $v$, the proof of the first formula follows. 
The second formula follows from the first and the 
combinatorial interpretation of the $h$-vector of a 
simple polytope $P$, mentioned in Section~\ref{sec:stack}.
\end{proof}

\begin{remark} \label{rem:sum-out} \rm
Since every edge of $\omega(P,f)$ has a unique initial 
vertex, the sum of the outdegrees $\out_f(v)$ of the 
vertices of $P$ in the directed graph $\omega(P, f)$ is 
equal to the number of edges of $P$.
\end{remark}

\begin{corollary} \label{thm:arb-simple}
For $m > d \ge 4$, the maximum number of $f$-arborescences
over all simple $d$-dimensional polytopes with $m$ facets
is achieved by the polar duals of neighborly polytopes and
is given by the formula
\[ \max \tau(P,f) = \prod\limits_{i=1}^{\lfloor \frac{d}{2}
  \rfloor} \, i^{{m-d+i-1 \choose i}}
	 \prod\limits_{i=0}^{\lfloor \frac{d-1}{2} \rfloor} \,
	 (d-i)^{{m-d+i-1 \choose i}}. \]
Similarly, the minimum number of $f$-arborescences in this
situation is achieved by the polar duals of stacked
polytopes and is given by the formula
\[ \min \tau(P,f) = d\cdot
  \left( (d-1)! \right)^{m-d}. \]
For $3$-dimensional simple polytopes $P$ with $m$ facets,
$\tau(P,f) = 3 \cdot 2^{m-3}$.
\end{corollary}

\begin{proof}
The case $d \ge 4$ follows from the last
sentence of Proposition~\ref{prop:product}, the upper
and lower bound theorems for the $h$-vector of a
simplicial polytope, discussed in Section~\ref{sec:stack},
and the formulas for the $h$-vectors of $d$-dimensional
neighborly and stacked simplicial polytope with $m$
vertices given there.
The case $d=3$ follows again from the second formula of
Proposition~\ref{prop:product}, since $h_0(P) = h_3(P) =
1$ and $h_1(P) = h_2(P) = m-3$ for every
$3$-dimensional simple polytope $P$ with $m$ facets.
\end{proof}

The following two statements apply to general polytopes. 
Combined with Corollary~\ref{thm:arb-simple}, they imply 
the results about $f$-arborescences stated in the 
introduction.

\begin{theorem}\label{thm:arb-max-all}
For $n > d \ge 3$, the maximum number of $f$-arborescences 
over all $d$-dimensional polytopes with $n$ vertices is 
achieved by the stacked polytope $X(n)$ for $d=3$ and by 
any 2-neighborly polytope for $d \ge 4$. This number is 
equal to $2 \cdot 3^{n-3}$ and $(n-1)!$ in the two cases, 
respectively.
\end{theorem}

\begin{proof}
Let us order the vertices $v_1, v_2,\dots,v_n$
of the $d$-dimensional polytope $P$ so that $f(v_1) \le 
f(v_2) \le \cdots \le f(v_n)$, where $v_n = v_{\max}$. 
Then, arcs 
of the directed graph $\omega(P,f)$ can only be pairs 
$(v_i, v_j)$ with $i < j$ and hence $\out_f(v_i) \le 
n-i$ for every $i \in [n]$. Thus, in view of 
Proposition~\ref{prop:product}, we get 
\[ \tau(P, f) \, = \prod\limits_{i=1}^{n-1} \out_f (v_i) 
   \, \le \prod\limits_{i=1}^{n-1} (n-i) \, = \, (n-1)! \]
and equality holds if and only if $P$ is 2-neighborly.

Since no such polytopes other than simplices exist in 
dimension $d=3$, this case has to be treated separately.
Setting $d_i = \out_f(v_i)$ for $i \in [n-1]$, we have 
positive integers $d_1, d_2,\dots,d_{n-1}$ such that 
$d_{n-1} = 1$ and $d_{n-2} \in \{1, 2\}$. Since $P$ can  
have no more than $3n-6$ edges, we have $d_1 + d_2 + 
\cdots + d_{n-1} \le 3n-6$ by Remark~\ref{rem:sum-out}. 
It is an elementary fact that, under these assumptions, 
the product $\tau(P, f) = d_1 d_2 \cdots d_{n-1}$ is 
maximized when $d_{n-1} = 1$, $d_{n-2} = 2$ and $d_i = 
3$ for every $i \in [n-3]$. Exactly that happens for 
the stacked polytope $X(n)$ and the proof follows.  
\end{proof}

\begin{theorem}\label{thm:arb-min-dim3-all}
For all $n \ge 4$, the minimum number of $f$-arborescences 
over all 3-dimensional polytopes with $n$ vertices is 
equal to $2 (n-1)$. This is achieved by any pyramid $P$ 
and any generic linear functional $f$ which takes 
its minimum value on $P$ at the apex.
\end{theorem}

\begin{proof}
As a simple application of 
Proposition~\ref{prop:product}, we have $\tau(P,f) = 
2(n-1)$ for every pyramid $P$ over an $(n-1)$-gon and 
every generic functional $f$ which takes its minimum 
value on $P$ at the apex.

We now consider any 3-dimensional polytope $P$ with 
$n$ vertices and any generic functional $f$ on $P$. 
We need to show that $\tau(P, f) \ge 2(n-1)$. We may 
linearly order the vertices $v_1, v_2,\dots,v_n$ of 
$P$ in the order of decreasing outdegree in the 
directed graph $\omega(P,f)$ and denote by $k$ the 
number of those vertices which have outdegree larger 
than one. Then, $k \ge 2$ and the respective outdegrees 
$d_1, d_2,\dots,d_n$ of $v_1, v_2,\dots,v_n$ satisfy 
$d_1, d_2,\dots,d_k \ge 2$, $d_n = 0$ and $d_i = 1$ 
for every other value of $i$. Letting $D_1, 
D_2,\dots,D_n$ be the degrees of $v_1, v_2,\dots,v_n$ 
in the undirected graph of $P$, respectively, we have 
$\tau(P,f) = d_1 d_2 \cdots d_k$ and 
\[ 2 \cdot \sum\limits_{i=1}^n d_i \, = \, 
   \sum\limits_{i=1}^n D_i \]
by Remark~\ref{rem:sum-out}. Clearly,
$D_i = d_i$ for one value of $i \in \{1, 2,\dots,k\}$
(the one corresponding to the source vertex), 
$D_i \ge d_i + 1$ for every other 
such value and $D_i \ge 3$ for all $k < i \le n$. 
These considerations result in the inequality $d_1 + d_2 
+ \cdots + d_k \ge n+1$ and thus, it remains to show 
that $d_1 d_2 \cdots d_k \ge 2(n-1)$ for every $k \ge 2$ 
and all $d_1, d_2,\dots,d_k \in \{2, 3,\dots,n-1\}$ 
summing at least to $n+1$. Indeed, from the inequality 
$ab > (a-1)(b+1)$ for integers $a \le b$, applied 
repeatedly when $b$ is the largest of $d_1, 
d_2,\dots,d_k$ and $a$ is any other number from these 
larger than 2, we get
\[ d_1 d_2 \cdots d_k \, \ge \, (d_1 + d_2 + \cdots + 
   d_k -2k + 2) \cdot 2^{k-1} \, \ge \, (n-2k+3) \cdot 
	 2^{k-1}. \]
Applying repeatedly the fact that $2m \ge m+2$ for $m 
\ge 2$, we conclude that $d_1 d_2 \cdots d_k \ge 2(n-1)$ 
and the proof follows. 
\end{proof}

More generally, for any $d \ge 3$, the $(d-2)$-fold pyramid 
$P$ over an $(n-d+2)$-gon has $n$ vertices and dimension $d$. 
Moreover, if $f$ is chosen so that every cone vertex has 
smaller objective value than any of the vertices of the 
$(n-d+2)$-gon, then the number of $f$-arborescences on $P$ 
is equal to $2 (n-1)(n-2) \cdots (n-d+2)$.   
\begin{question} \label{que:arb-min-all}
What is the minimum number of $f$-arborescences over all 
$d$-dimensional polytopes with $n$ vertices, for $d \ge 4$? 
Does it equal $2 (n-1)(n-2) \cdots (n-d+2)$ for all $n > 
d \ge 4$?
\end{question}

\section{On the number of monotone paths}
\label{sec:paths}

This section investigates the smallest and largest 
possible number of $f$-monotone paths on polytopes. 
For notational 
convenience, we let $v_0, v_1,\dots,v_n$ be the 
vertices of a polytope $P$, linearly ordered so that 
$f(v_0) < f(v_1) < \cdots < f(v_n)$, as in 
Section~\ref{sec:graph}. We recall that $\mu(P, f)$ 
denotes the number of $f$-monotone paths on $P$ and 
that we refer to general directed paths in 
$\omega(P, f)$ as partial $f$-monotone paths, i.e., 
they may start or end at vertices other than 
$v_{\min}$ or $v_{\max}$. 

The following formula is the key to most results in this 
section.

\begin{lemma} \label{lem:numofpaths}
The number of $f$-monotone paths on $P$ can be expressed
as
\[ \mu(P, f) \, = \, 1 + \sum_{k=0}^{n-1} \, (d_k - 1) 
   \mu_k (P, f), \]
where $d_k = \out_f(v_k)$ is the outdegree of $v_k$ in 
$\omega (P, f)$ and $\mu_k(P, f)$ stands for the number 
of partial $f$-monotone paths on $P$ with initial vertex 
$v_0$ and terminal vertex $v_k$. 
\end{lemma}

\begin{proof}
Let $P(t) = f^{-1}(t) \cap P$ be the fibers of the map 
$f: P \to \RR$, as in Section~\ref{sec:graph}, and $t_i 
= f(v_i)$ for $0 \le i \le n$. For $0 \le k \le n-1$ let 
$\hH_k(P, f)$ be the set of partial $f$-monotone 
paths on $P$ having initial vertex $v_0$ and ending in 
the fiber $P(t)$ with $t_k < t < t_{k+1}$. Formally, 
these are essentially the nodes of the inverse limit of 
the part 
\[  G_{0,1} \ \overset{\alpha_1}{\toto} \ G_1 \ 
\overset{\beta_1}{\getto} \ G_{1,2} \ \overset{\alpha_2}
{\toto} \ G_2 \ \overset{\beta_2}{\getto} \ \cdots \   
\overset{\alpha_k}{\toto} \ G_k \ \overset{\beta_k}
{\getto} \ G_{k,k+1} \]
of the diagram~(\ref{eq:main-diagram}). Let $\eta_k
(P, f)$ be the number of these partial $f$-monotone 
paths. We claim that 
\begin{equation}
\label{eq:mu-prime}
\eta_k(P, f) - \eta_{k-1}(P, f) \, = \, 
   (d_k - 1) \mu_k (P, f)  
\end{equation}
for every $k \in [n-1]$. Since $\eta_0(P, f) = 
\out_f(v_0) = d_0$ and $\mu_0(P, f) = 1$, this implies 
that
\[ \eta_k(P, f) \, = \, 1 + \sum_{i=0}^k \, (d_i - 1) 
   \mu_i (P, f) \]
for $0 \le k \le n-1$. Since $\eta_{n-1}(P, f) = 
\mu_n(P, f) = \mu(P, f)$, the desired formula follows 
as the special case $k = n-1$ of this equation. 

To verify~(\ref{eq:mu-prime}), let $\varphi_k: \hH_k
(P, f) \to \hH_{k-1}(P, f)$ be the natural map
obtained by restriction of diagrams. More intuitively, 
$\varphi_k(\gamma)$ is obtained from $\gamma \in 
\hH_k(P,f)$ by removing its last edge. Paths in 
$\hH_{k-1}(P, f)$ and $\hH_k(P, f)$ 
either pass through vertex $v_k$ or not, depending on 
whether or not their last edge maps to $v_k$ under the
map $\alpha_k$ or $\beta_k$, respectively. Clearly, for
every $\delta \in \hH_{k-1}(P, f)$ which passes 
through $v_k$ there are exactly $d_k$ paths $\gamma \in 
\hH_k(P, f)$ such that $\varphi_k(\gamma) = 
\delta$, obtained by choosing an edge of $\omega(P,f)$
coming out of $v_k$ and attaching it to $\delta$, while 
for every $\delta \in \hH_{k-1}(P, f)$ which does 
not pass through $v_k$ there is a unique path $\gamma 
\in \hH_k(P, f)$ such that $\varphi_k(\gamma) = 
\delta$. These observations imply directly 
Equation~(\ref{eq:mu-prime}) and the proof follows.
\end{proof}

Recall that the Tribonacci sequence $(T_n)$ is 
defined by the recurrence relation $T_0 = T_1 = 1$, 
$T_2 = 2$ and $T_n = T_{n-1} + T_{n-2} + T_{n-3}$ for 
$n \ge 3$.
\begin{theorem} \label{thm:path-max}
The maximum number of $f$-monotone paths over all 
3-dimensional polytopes with $n+1$ vertices is equal to 
the $n$th Tribonacci number $T_n$ for every $n \ge 3$. 
This is achieved by the stacked polytope $X(n)$.
\end{theorem}

\begin{proof}
We proceed by induction on $n$. The result holds for 
$n=3$, since there are exactly $T_3 = 4$ monotone 
paths on any 3-dimensional simplex. We assume that it 
holds for integers less than $n$ and consider a 
3-dimensional polytope $P$ with $n+1$ vertices $v_0, 
v_1,\dots,v_n$, linearly ordered as in the beginning 
of this section by a generic functional $f$. 

We wish to apply Lemma~\ref{lem:numofpaths}. Since partial
$f$-monotone paths on $P$ with initial vertex $v_0$ and 
terminal vertex $v_k$ are $f$-monotone paths on the convex
hull of $v_0, v_1,\dots,v_k$, we have $\mu_k (P, f) \le T_k$ 
for $k \in \{3, 4,\dots,n-1\}$ by the induction hypothesis. 
Since this bound holds trivially for $k \in \{0, 1, 2\}$ 
as well, from Lemma~\ref{lem:numofpaths} we get 
\[ \mu(P, f) \, \le \, 1 + \sum_{k=0}^{n-1} \, (d_k - 1) 
   T_k. \]
To bound the right-hand side, we note that
\[ d_{n-k} + d_{n-k+1} + \cdots + d_{n-1} \le 3k-3 \]
for $k \in \{2, 3,\dots,n-1\}$, since $d_{n-k} + d_{n-k+1} 
+ \cdots + d_{n-1}$ is equal to the number of edges of $P$ 
connecting vertices $v_{n-k}, v_{n-k+1},\dots,v_n$ and 
hence to the number of edges of a planar simple graph with 
$k+1$ vertices. From these inequalities and the trivial  
one $d_{n-1} \le 1$, and setting $T_{-1} := 0$, we get
\begin{align*}
\sum_{k=0}^{n-1} d_k T_k \ & \ = \
\sum_{k=1}^n (d_{n-1} + d_{n-2} + \cdots + d_{n-k}) 
             (T_{n-k} - T_{n-k-1}) \\
    & \ \le \ (T_{n-1} - T_{n-2}) + (3k-3) \sum_{k=2}^n 
		  (T_{n-k} - T_{n-k-1}) \\
    & \ = \ T_{n-1} + 2T_{n-2} + 3 T_{n-3} + 3 T_{n-4} + 
		  \cdots + 3 T_0 \\
		& \ = \ \sum_{k=1}^n T_k,
\end{align*}
where the last equality comes from summing the 
recurrence $T_k = T_{k-1} + T_{k-2} + T_{k-3}$ for $k \in 
[n]$. We conclude that
\[ \mu(P, f) \, \le \, 1 + \sum_{k=0}^{n-1} \, (d_k - 1) 
   T_k \, = \, 1 + \sum_{k=0}^{n-1} d_k T_k - 
	 \sum_{k=0}^{n-1} T_k \, \le \, T_n. \]
This completes the induction. 

Finally, it is straightforward to verify that the number 
of $f$-monotone paths on $X(n+1)$ satisfies the Tribonacci 
recurrence (or alternatively, that all inequalities hold 
as equalities in the previous argument) and is thus equal 
to $T_n$ for every $n$.
\end{proof}

\begin{remark} \label{thm:path-max-alldim}
\rm The number of $f$-monotone paths on a polytope $P$ 
with $n+1$ vertices is no larger than the number of subsets 
of its vertex set containing the source and the sink and 
hence at most $2^{n-1}$. Equality holds exactly when
$P$ is 2-neighborly, meaning that the 1-skeleton of $P$
is the complete graph on $n+1$ vertices, since then every 
such subset is the vertex set of an $f$-monotone path on 
$P$. As a result, the maximum number of $f$-monotone paths 
over all $d$-dimensional polytopes with $n+1$ vertices is 
equal to $2^{n-1}$ for all $n \ge d \ge 4$. 
\qed
\end{remark}

The following statement completes the proof of 
the results about the number of $f$-monotone paths, 
stated in the introduction.

\begin{theorem}\label{thm:path-min-dim3}
The minimum number of $f$-monotone paths over all 
$3$-dimensional polytopes with $n$ vertices is equal to 
$\left\lceil{n/2}\right \rceil + 2$. This is achieved 
by prisms, when $n$ is even, and by wedges of polygons over 
a vertex, when $n$ is odd. 

In particular, prisms minimize the number of $f$-monotone 
paths over all simple polytopes of dimension three with 
given number of vertices.
\end{theorem}

\begin{proof}
Applying Lemma~\ref{lem:numofpaths} and 
noting that $\mu_k(P, f) \ge 1$ for every $k$, we get
\[ \mu(P, f) \, \ge \, 1 + \sum_{k=0}^{n-2} \, (d_k - 1) \, 
   = \, \sum_{k=0}^{n-2}d_k - n + 2. \] 
Since $\sum_{k=0}^{n-2}d_k$ is equal to the number of 
edges of $P$ (see Remark~\ref{rem:sum-out}), which is  
bounded below by 
$\left\lceil{3n/2}\right \rceil$, it follows that 
$\mu(P, f) \ge \left\lceil{n/2}\right \rceil+ 2$. It is 
straightforward to verify that prisms achieve the minimum 
when $n$ is even and wedges of polygons over a vertex 
(obtained from prisms by identifying two vertices at 
different levels which are connected by an edge) 
achieve the minimum when $n$ is odd. 
\end{proof}

The lower bound for the number of $f$-monotone paths 
in any dimension, given in the following statement, is 
not expected to be tight.

\begin{proposition} \label{cor:path-min-alldim}
The number of $f$-monotone paths on any polytope of 
dimension $d$ with $n$ vertices is bounded below by 
$\left\lceil{dn/2}\right \rceil-n + 2$.
\end{proposition}

\begin{proof}
Once again, this follows from the inequality 
$\sum_{k=0}^{n-2} d_k \ge \left\lceil{dn/2}\right 
\rceil$ and Lemma~\ref{lem:numofpaths}.
\end{proof}

We end this section with a conjecture for the maximum 
number of monotone paths on simple 3-dimensional polytopes. 
The proposed maximum can be achieved by wedges of polygons 
over an edge whose vertices are the source and the sink, and all vertices of the polytope lie on a monotone 
path. We recall that the Fibonacci sequence $(F_n)$ is 
defined by the recurrence $F_1 = F_2 = 1$ and $F_n = 
F_{n-1} + F_{n-2}$ for $n \ge 2$.

\begin{conjecture}\label{conj:path-max-simple}
We have $\mu(P, f) \le F_{n+2} + 1$ for every simple 
3-dimensional polytope $P$ with $2n$ vertices.
\begin{figure*}[h]
\includegraphics[scale = 0.35]{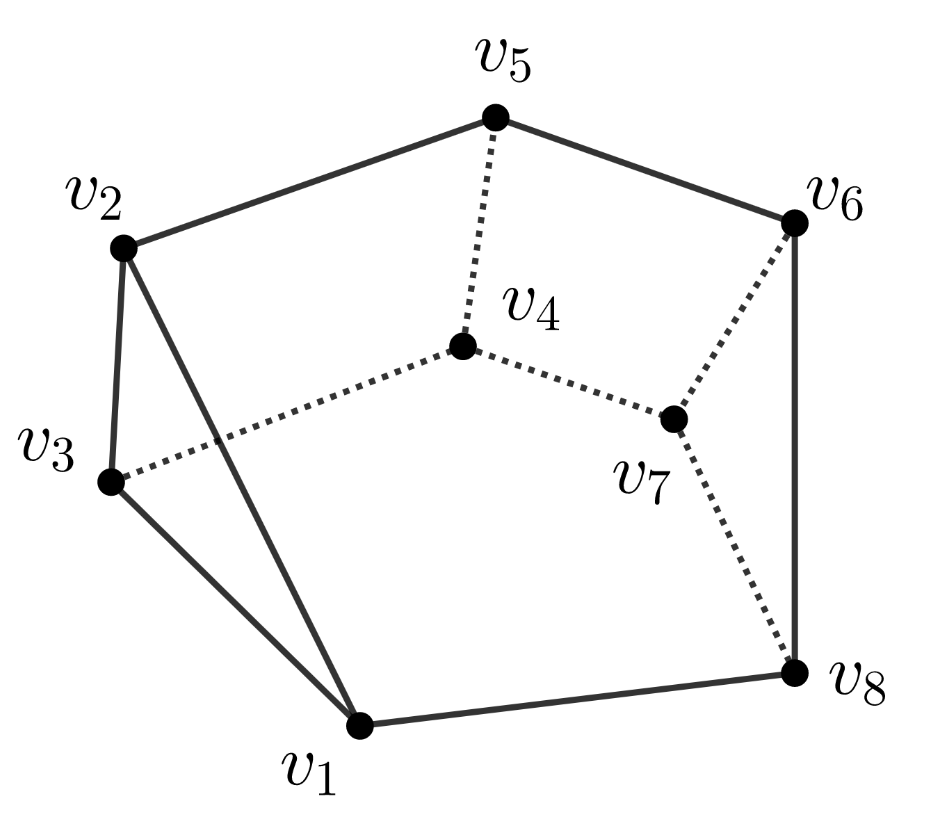}
\caption{An example of a polytope on 8 vertices 
conjectured to be the maximizer of the number of 
monotone paths among simple polytopes. $f(v_1) < f(v_2) < \cdots < f(v_8)$.} 
\end{figure*}
\end{conjecture}

The argument in the proof of Theorem~\ref{thm:path-max} 
yields the following weaker result. 

\begin{proposition} \label{cor:path-max-simple}
We have $\mu(P, f) \leq 2F_n$ for every 3-dimensional 
simple polytope $P$ with $n+1$ vertices and every generic
linear functional $f$ on $P$. 
\end{proposition}

\begin{proof}
Let $(a_n)$ be the 
sequence of numbers defined by the recurrence relation 
$a_0 = a_1 = 1$, $a_2 = 2$, $a_3 = 4$ and $a_n = a_{n-1} 
+ a_{n-2}$ for $n \ge 4$. Note that $a_n = 2F_n$ for $n \geq 2$. We mimick the proof of Theorem~\ref{thm:path-max} to show that $\mu(P, f) \leq a_n$ for all $n \ge 3$. For the 
inductive step, since $P$ is simple, we have $d_0 = 3$, 
$d_1,d_2,\dots,d_{n-2} \le 2$ and $d_{n-1}= 1$ and compute
that 
\begin{align*}
\mu(P, f) & \ \le \ 1 + \sum_{k=0}^{n-1} d_k a_k - 
             \sum_{k=0}^{n-1} a_k \ \le \ 1 + a_{n-2} + 
						 a_{n-3} + \ldots + a_1 + 2a_0 \\ 
& \ \le \ a_{n-1} + a_{n-2} \ = \ a_n,
\end{align*}
since $a_{n-1} = 1 + a_{n-3} + \cdots + a_1 + 2a_0$.
\end{proof}

\section{On the diameter of monotone path graphs}
\label{sec:diameter}

The main goal of this section is to prove 
Theorem~\ref{thm:diam-max}. 

The lower bound 
of~(\ref{eq:diam-max}) for the maximum diameter follows from 
Lemma~\ref{lem:diam-stack}. The upper bound will be deduced from 
the following result. Clearly, given a polytope $P$ and a generic 
linear functional $f$, every $f$-monotone path on $P$ meets each 
of the fibers $f^{-1}(t) \cap P$, where $t \in f(P)$, in a unique 
point. For $f$-monotone paths $\gamma$ and $\gamma'$ on $P$, 
let us denote by $\nu(\gamma, \gamma')$ the number of connected 
components of the set of values $t \in f(P)$ for which $\gamma$ 
and $\gamma'$ disagree on $f^{-1}(t) \cap P$. For example, 
for the two monotone paths, say $\gamma$ and $\gamma'$, shown 
on Figure~\ref{colored-paths} we have $\nu(\gamma, \gamma') = 4$. 
Note that $\nu(\gamma, \gamma') = 0 \Leftrightarrow 
\gamma = \gamma'$.

\begin{theorem} \label{thm:upper-dist}
Let $P$ be a 3-dimensional polytope and $f$ be a generic linear
functional on $P$. The distance between any two $f$-monotone 
paths $\gamma$ and $\gamma'$ in the graph $G = G(P,f)$ satisfies 
\begin{equation}
\label{eq:upper:dist}
d_G (\gamma,\gamma') \le \frac{\nu(\gamma, \gamma')}{2} 
\cdot f_2(P),
\end{equation}
where $f_2(P)$ is the number of 2-dimensional faces of $P$.
\end{theorem}

\begin{remark}\rm 
Theorem \ref{thm:upper-dist} gives a diameter bound for all three-dimensional polytopes. Cordovil and Moreira had studied bounds for three-dimensional zonotopes and rank-three oriented matroids \cite{cordovil+moreira}, which they gave in terms of the dual pseudo-line arrangements. 
\end{remark}

We will first state a technical result (see 
Proposition~\ref{prop:polygon-path}) which constructs a 
walk in $G(P,f)$ between two monotone paths $\gamma$ and 
$\gamma'$ with the required properties from walks on the 
fibers, assuming that the latter satisfy certain necessary 
compatibility conditions. To allow for all possible ways
that $\gamma$ and $\gamma'$ may intersect each other, we 
consider the following general situation.
Let $\fF$ be a connected polygonal
complex in $\RR^d$ having vertices $v_0, v_1,\dots,v_n$ and 
$f: \RR^d \to \RR$ be a linear functional such that $f(v_0) 
< f(v_1) < \cdots < f(v_n)$. The graph of $f$-monotone paths 
on $\fF$, denoted by $G(\fF, f)$, having initial vertex $v_0$ 
and terminal vertex $v_n$, can be defined with adjacency 
given by polygon flips just as in the special case in which 
$\fF$ is the 2-skeleton of a convex polytope (see 
Section~\ref{sec:graph}). Alternatively, and in order to 
relate it to the graphs of the fibers of $f$, we may view 
$G(\fF,f)$ as the inverse limit associated to a diagram

\begin{equation}
\label{eq:main-diagram2}
G_{0,1} \ \overset{\alpha_1}{\toto} \ G_1 \ 
\overset{\beta_1}{\getto} \ G_{1,2} \ \overset{\alpha_2}{\toto} 
\ G_2 \ \overset{\beta_2}{\getto} \ \cdots \ 
\overset{\beta_{n-2}}{\getto} \ G_{n-2,n-1} \   
\overset{\alpha_{n-1}}{\toto} \ G_{n-1} \ \overset{\beta_{n-1}}
{\getto} \ G_{n-1,n}
\end{equation}
 
\bigskip
\noindent
of graphs and simplicial maps. This is defined as in 
Section~\ref{sec:graph} provided the fiber $f^{-1}(t) \cap P$ 
is replaced with $f^{-1}(t) \cap \|\fF\|$, where $\|\fF\|$ is 
the polyhedron (union of faces) of $\fF$. Thus, the $G_i$ and 
$G_{i,i+1}$ are graphs of (one-dimensional) fibers $f^{-1}(t) 
\cap \|\fF\|$ and the $\alpha_i$ and $\beta_i$ are natural 
degeneration maps.

Given an $f$-monotone path $\gamma$ on $\fF$ and $i \in [n]$,
let us denote by $\pi_i(\gamma)$ the node of $G_{i-1,i}$ in 
which the union of the edges of $\gamma$ intersects the 
corresponding fiber $f^{-1}(t) \cap \|\fF\|$. Then, $\pi_i: 
G(\fF, f) \to G_{i-1,i}$ is a simplicial map. Given a walk 
$\pP$ in a graph $G$, thought of as a sequence of edges, and 
a simplicial map $\varphi: G \to H$ of graphs, let us denote 
by $\varphi(\pP)$ the walk in $H$ which is formed by the images 
of the edges of $\pP$ under $\varphi$, disregarding those edges 
of $\pP$ which are contracted to a node by $\varphi$.

\begin{proposition} \label{prop:polygon-path}
Let $\gamma$ and $\delta$ be $f$-monotone paths on $\fF$.
Suppose that for every $i \in [n]$ there exists a walk $\pP_i$ 
in $G_{i-1,i}$ with initial node $\pi_i(\gamma)$ and terminal
node $\pi_i(\delta)$ which traverses each edge in $G_{i-1,i}$
exactly once, so that 
\begin{equation} \label{eq:Pcondition}
\alpha_i(\pP_i) = \beta_i(\pP_{i+1})
\end{equation} 
for every $i \in [n-1]$. Then, there exists a walk $\pP$ in 
$G(\fF, f)$ with initial node $\gamma$ and terminal node 
$\delta$ which traverses each 2-dimensional face of $\fF$ 
exactly once, such that $\pi_i(\pP) = \pP_i$ for every 
$i \in [n]$. 
\end{proposition}

We first illustrate the proposition with an important 
special case and then use it to prove 
Theorem~\ref{thm:upper-dist}.

\begin{example} \label{ex:disjoint-paths} \rm
To motivate the proof of Theorem \ref{thm:upper-dist}, 
consider the special case in 
which the monotone paths $\gamma$ and $\gamma'$ do not 
have common vertices, other than those on which $f$ attains its 
minimum and maximum value on $P$. Then, $\nu(\gamma, \gamma') = 1$ 
and the edges of $\gamma$ and $\gamma'$ form a simple cycle $C$ which 
divides the boundary of $P$ into two closed balls, say $B^+$ 
and $B^-$, having common boundary $C$. Let $\fF^+$ 
and $\fF^-$ be the two subcomplexes of the boundary complex 
of $P$ which correspond to these balls. We wish to show that 
for each $\varepsilon \in \{+, -\}$, there exists a walk 
in $G(P,f)$ joining $\gamma$ and $\gamma'$ which traverses each 
2-dimensional face of $\fF^\varepsilon$ exactly once. This 
would imply the desired bound for $d_G (\gamma,\gamma')$. 
Such a walk must traverse every edge of each fiber 
$f^{-1}(t) \cap B^\varepsilon$ exactly once and thus induce 
walks on these fibers with the same property. 

Let us consider the diagram (\ref{eq:main-diagram2}) for the 
polygonal complex $\fF^\varepsilon$. Clearly, the fiber $f^{-1}(t) \cap 
\partial P$ is the boundary of a polygon for every interior point $t$ 
of the interval $f(P)$, where $\partial P$ 
denotes the boundary of $P$. Since, by the $f$-monotonicity 
of $\gamma$ and $\gamma'$, this fiber intersects the cycle 
$C$, which is the boundary of the ball $B^\varepsilon$, in 
exactly two points, its intersection with $B^\varepsilon$ 
must be homeomorphic to a line segment. Thus, all 
graphs appearing in the diagram (\ref{eq:main-diagram2}) 
for $\fF^\varepsilon$ are path graphs, where $G_{i-1,i}$ 
has endpoints $\pi_i(\gamma)$ and $\pi_i(\gamma')$ for every 
$i \in [n]$. As a result, there are unique walks $\pP_i$, as 
in the statement of 
Proposition~\ref{prop:polygon-path}, where condition
(\ref{eq:Pcondition}) holds by the degeneration of fibers. Thus, 
Proposition~\ref{prop:polygon-path} implies the existence 
of a walk in $G(\fF^\varepsilon,f)$ with initial node 
$\gamma$ and terminal node $\gamma'$ which traverses each 
2-dimensional face of $\fF^\varepsilon$ exactly once.
\qed
\end{example}

\medskip
\noindent
\textit{Proof of Theorem~\ref{thm:upper-dist}.}
We first observe that it suffices to prove the special case 
$\nu(\gamma, \gamma') = 1$. Indeed, given $f$-monotone paths 
$\gamma$ and $\gamma'$ on $P$ and setting $\nu = \nu(\gamma, 
\gamma')$, it is straightforward to define $f$-monotone paths 
$\gamma = \gamma_0$, $\gamma_1,\dots,\gamma_\nu = \gamma'$ on 
$P$ satisfying $\nu(\gamma_{i-1}, \gamma_i) = 1$ for every 
$i \in [\nu-1]$. Then, the triangle inequality and the special 
case imply that
\[ d_G(\gamma, \gamma') \le \sum_{i=1}^\nu d_G (\gamma_{i-1}, 
   \gamma_i) \le \nu \cdot \frac{f_2(P)}{2}, \]
as claimed by (\ref{eq:upper:dist}).

So let $\gamma, \gamma'$ be $f$-monotone paths on $P$ 
such that $\nu(\gamma, \gamma') = 1$. Let $u$ and $v$ 
be their unique common vertices, satisfying $f(u) < f(v)$, 
for which $\gamma$ and $\gamma'$ disagree on each fiber 
$f^{-1}(t) \cap P$ with $f(u) < t < f(v)$ and agree on 
the other fibers; in the special
case of Example~\ref{ex:disjoint-paths}, $u$ and $v$ 
are the unique vertices $v_{\min}$ and $v_{\max}$ on which 
$f$ attains its minimum and maximum value on $P$, 
respectively. As in that special case, the edges 
of $\gamma$ and $\gamma'$ joining $u$ and $v$ form a simple 
cycle $C$ which divides the 2-dimensional sphere $\partial P$ 
into two closed 2-dimensional balls $B^+$ and $B^-$ having
common boundary $C$. Moreover, the $f$-monotonicity 
of $\gamma$ and $\gamma'$ implies that for each $\varepsilon 
\in \{+, -\}$ and every interior point $t$ of the interval 
$f(B^\varepsilon)$, the fiber $f^{-1}(t) \cap B^\varepsilon$ 
is homeomorphic to a line segment or a circle.
We wish to apply Proposition~\ref{prop:polygon-path} to the 
subcomplex $\fF^\varepsilon$ of the boundary complex of $P$ 
which corresponds to $B^\varepsilon$.

We claim that there 
exist unique walks $\pP_i$ satisfying the assumptions of the
proposition. Indeed, according to our previous discussion, 
every graph $G_{i-1,i}$ appearing in the diagram 
(\ref{eq:main-diagram2}) for $\fF^\varepsilon$ is either a 
path graph, with endpoints $\pi_i(\gamma)$ and $\pi_i(\gamma')$, 
or a cycle. As a result, there exists a unique walk $\pP_i$ 
in $G_{i-1,i}$ with initial node $\pi_i(\gamma)$ and terminal 
node $\pi_i(\gamma')$ which traverses each edge in $G_{i-1,i}$ 
exactly once, if the latter is a path graph, and exactly two 
such walks, corresponding to the two possible orientations
of $G_{i-1,i}$, if the latter is a cycle. There are the 
following cases, illustrated in Example~\ref{ex:cycle-fibers}, 
to consider:

\medskip
\noindent
\textbf{Case 1:} The relative interior of $B^\varepsilon$
contains neither $v_{\min}$ nor $v_{\max}$. Then, all the 
$G_{i-1,i}$ are path graphs and conditions 
(\ref{eq:Pcondition}) hold by degeneration of fibers, as in 
the special case $u = v_{\min}$ and $v = v_{\max}$ of 
Example~\ref{ex:disjoint-paths}. 

\noindent
\textbf{Case 2:} The relative interior of $B^\varepsilon$
contains exactly one of $v_{\min}$ and $v_{\max}$, say 
$v_{\min}$. Then, the $G_{i-1,i}$ associated to fibers 
$f^{-1}(t) \cap B^\varepsilon$ with $t < f(u)$ are cycles 
and all others are path graphs which degenerate to cycles
as the value of $f$ approaches $f(u)$. Clearly, 
the cycles can be uniquely oriented, so that the resulting 
walks $\pP_i$ satisfy conditions (\ref{eq:Pcondition}). 

\noindent
\textbf{Case 3:} The relative interior of $B^\varepsilon$
contains both $v_{\min}$ and $v_{\max}$. Then, the $G_{i-1,i}$ 
associated to fibers $f^{-1}(t) \cap B^\varepsilon$ with 
$f(u) < t < f(v)$ are path graphs and the rest are cycles 
which can be uniquely oriented, so that the resulting walks 
$\pP_i$ satisfy conditions (\ref{eq:Pcondition}).

\medskip
Thus, Proposition~\ref{prop:polygon-path} applies in all 
cases and we may conclude that $d_G(\gamma, \gamma') \le f_2
(\fF^\varepsilon)$ for each $\varepsilon \in \{+, -\}$. 
Hence,
\[ d_G(\gamma, \gamma') \le \frac{f_2(\fF^+) + f_2(\fF^-)}{2} 
   = f_2(P)/2 \]
and the proof follows.
\qed

\begin{example} \label{ex:cycle-fibers}
\rm Let $P = X(10)$ be the stacked polytope shown in 
Figure~\ref{Xtacked}. The following two situations 
illustrate the three cases within the proof 
of Theorem \ref{thm:upper-dist}.

\medskip
\noindent
(a) Consider the $f$-monotone paths on $P$
\begin{align*}
\gamma \, & \, = \, (v_1, v_3, v_6, v_9, v_{10}), \\
\gamma' \, & \, = \, (v_1, v_3, v_5, v_8, 
                                   v_9, v_{10}),
\end{align*}
presented as sequences of vertices. Then, the cycle $C$
has edges with vertex sets $\{v_3, v_5\}$, $\{v_5, v_8\}$, 
$\{v_8, v_9\}$, $\{v_6, v_9\}$ and $\{v_3, v_6\}$, and 
one of the $\fF^\varepsilon$ consists of the faces of 
the facets of $P$ with vertex sets $\{v_3, v_5, v_6\}$, 
$\{v_5, v_6, v_8\}$ and $\{v_6, v_8, v_9\}$ and falls in 
the first case of the proof, while the other consists of 
the faces of the remaining thirteen facets of $P$ and 
falls in the third case. Three flips are needed to reach 
$\gamma'$ from $\gamma$ across $\fF^\varepsilon$ in
the former case, and thirteen flips in the latter.

\noindent
(b) Consider also the $f$-monotone paths
\begin{align*}
\gamma \, & \, = \, (v_1, v_3, v_6, v_9, v_{10}), \\
\gamma'' \, & \, = \, (v_1, v_3, v_4, v_5, v_8, v_9, v_{10}).
\end{align*}
Now $C$ has six edges with vertex sets $\{v_3, v_4\}$, 
$\{v_4, v_5\}$, $\{v_5, v_8\}$, $\{v_8, v_9\}$, $\{v_6, v_9\}$ 
and $\{v_3, v_6\}$, and one of the $\fF^\varepsilon$ consists of 
the faces of the facets of $P$ with vertex sets $\{v_1, v_2, v_3\}$, 
$\{v_1, v_2, v_4\}$, $\{v_1, v_3, v_4\}$, $\{v_2, v_3, v_5\}$, 
$\{v_2, v_4, v_5\}$, $\{v_3, v_5, v_6\}$, $\{v_5, v_6, v_8\}$ 
and $\{v_6, v_8, v_9\}$, while the other consists of the 
faces of the remaining eight facets of $P$. Both fall in
the second case of the proof. The fibers $f^{-1}(t) \cap 
B^\varepsilon$ are path graphs for $f(v_3) < t < f(v_9)$ in
either case, and cycles for $t \le f(v_3)$ or $t \ge f(v_9)$ in
the two cases, respectively. 
\qed
\end{example}

\medskip
\noindent
\textit{Proof of Theorem~\ref{thm:diam-max}.}
As we have already mentioned, the lower bound 
of~(\ref{eq:diam-max}) follows from 
Lemma~\ref{lem:diam-stack}. The upper bound follows from 
Theorem~\ref{thm:upper-dist} and the obvious inequalities 
$\nu(\gamma, \gamma') \le \lfloor (n-1)/2\rfloor$ and 
$f_2(P) \le 2n-4$.
\qed

\begin{question} \label{que:lower-diam-dim3}
What is the exact value of the maximum diameter in 
Theorem~\ref{thm:diam-max}? In particular, is it equal 
to the lower bound given there for every $n$?
\end{question}

\medskip
\noindent
\textit{Proof of Proposition~\ref{prop:polygon-path}.}
Consider indices $0 < k \le m \le \ell < n$ and denote by $K$ 
and $L$ the graphs of partial $f$-monotone paths on $\fF$ which 
arise as inverse limits of the subdiagrams 

\begin{equation}
\label{eq:K-diagram}
G_{k-1,k} \ \overset{\alpha_k}{\toto} \ G_k \ 
\overset{\beta_k}{\getto} \ G_{k,k+1} \ 
\overset{\alpha_{k+1}}{\toto} \ \cdots \ 
\overset{\alpha_{m-1}}{\toto} \ G_{m-1} \ \overset{\beta_{m-1}}
{\getto} \ G_{m-1,m}
\end{equation}
and
\begin{equation}
\label{eq:L-diagram}
G_{m,m+1} \ \overset{\alpha_{m+1}}{\toto} \ G_{m+1} \ 
\overset{\beta_{m+1}}{\getto} \ \cdots \ 
\overset{\beta_{\ell-1}}{\getto} \ G_{\ell-1,\ell} \ 
\overset{\alpha_\ell}{\toto} \ G_\ell \ \overset{\beta_{\ell}}
{\getto} \ G_{\ell,\ell+1}
\end{equation}

\bigskip
\noindent
of~(\ref{eq:main-diagram2}), respectively. Let us call a 
\emph{polygon} any 2-dimensional face of $\fF$ which intersects 
the fiber $f^{-1}(t) \cap \|\fF\|$ for some $t_{k-1} < t < t_m$
in the case of~(\ref{eq:K-diagram}) and any 2-dimensional face 
of $\fF$ which intersects the fiber $f^{-1}(t) \cap \|\fF\|$ for 
some $t_m < t < t_{\ell+1}$ in the case of~(\ref{eq:L-diagram}).
Thus, the polygons are exactly the 2-dimensional faces of $\fF$
in the case of the entire diagram~(\ref{eq:main-diagram2}) and 
are in one-to-one correspondence with the edges of $G_{m-1,m}$
in the special case $k=m$ of~(\ref{eq:K-diagram}). 
Define similarly the graph $H$ of partial $f$-monotone paths on 
$\fF$ and its polygons from the subdiagram 

\begin{equation}
\label{eq:H-diagram}
G_{k-1,k} \ \overset{\alpha_k}{\toto} \ \cdots \ 
\overset{\beta_{m-1}}{\getto} \ G_{m-1,m} 
\overset{\alpha_m}{\toto} \ G_m \ \overset{\beta_m}{\getto} \ 
G_{m,m+1} \ \overset{\alpha_{m+1}}{\toto} \ \cdots \ 
\overset{\beta_{\ell}}{\getto} \ G_{\ell,\ell+1}
\end{equation}

\bigskip
\noindent
of~(\ref{eq:main-diagram2}) and note that there are natural 
restriction maps $\pi_K: G(\fF, f) \to K$, $\pi_L: G(\fF, f) 
\to L$ and $\pi_H: G(\fF, f) \to H$. 

We assume that there exist 
a walk $\qQ$ in $K$ with initial node $\pi_K(\gamma)$ and 
terminal node $\pi_K(\delta)$ which traverses each polygon 
of~(\ref{eq:K-diagram}) exactly once and a walk $\rR$ in $L$ 
with initial node $\pi_L(\gamma)$ and terminal node $\pi_L
(\delta)$ which traverses each polygon of~(\ref{eq:L-diagram}) 
exactly once, such that $\pi_i(\qQ) = \pP_i$ for $k \le i \le 
m$ and $\pi_i(\rR) = \pP_i$ for $m < i \le \ell+1$. As a consequence, 
there exists a walk $\pP$ in $H$ with initial node 
$\pi_H(\gamma)$ and terminal node $\pi_H(\delta)$ which 
traverses each polygon of~(\ref{eq:H-diagram}) exactly once,
such that $\pi_i(\pP) = \pP_i$ for $k \le i \le \ell+1$.
The proposition then follows by applying the claim several 
times, for instance when $k=1$ and $m = \ell$, for $m \in 
[n-1]$.

To prove the claim, we only need to patch $\qQ$ and $\rR$ 
along the walk $\alpha_m(\pP_m) = \beta_m(\pP_{m+1})$ in 
$G_m$. Any two nodes $\zeta$ of $K$ and $\eta$ of $L$ 
produce by concatenation a node $\zeta \ast \eta$ of $H$, 
provided that the terminal edge of $\zeta$ and the initial 
edge of $\eta$ have equal images under $\alpha_m$ and 
$\beta_m$, respectively. Let $\zeta_0,
\zeta_1,\dots,\zeta_q$ be the successive nodes of $\qQ$ and 
$\eta_0, \eta_1,\dots,\eta_r$ be the successive nodes of 
$\rR$. By our assumptions, we have $\zeta_0 \ast \eta_0 =
\pi_K(\gamma) \ast \pi_L(\gamma) = \pi_H(\gamma)$ and 
$\zeta_q \ast \eta_r = \pi_K(\delta) \ast \pi_L(\delta) = 
\pi_H(\delta)$. We define $\pP$ to have nodes of the form 
$\zeta_i \ast \eta_j$, starting with $\zeta_0 \ast \eta_0$,
so that the node immediately following $\zeta_i \ast \eta_j$ 
is 

\begin{equation} \label{eq:steps}
\begin{cases}
\zeta_{i+1} \ast \eta_j, & \text{if well defined,} \\
\zeta_i \ast \eta_{j+1}, & \text{if well defined but 
      $\zeta_{i+1} \ast \eta_j$ is not,} \\
\zeta_{i+1} \ast \eta_{j+1}, & \text{otherwise.}			
\end{cases} 
\end{equation}

\bigskip
\noindent
We leave to the reader to verify that, because $\alpha_m
(\pP_m) = \beta_m(\pP_{m+1})$, this is a well defined walk 
in $H$ with initial node $\zeta_0 \ast \eta_0 = \pi_H
(\gamma)$ and terminal node $\zeta_q \ast \eta_r = \pi_H
(\delta)$. By construction, we have $\pi_i(\pP) = \pi_i
(\qQ)$ for $k \le i \le m$ and $\pi_i(\pP) = \pi_i(\rR)$ 
for $m < i \le \ell+1$, and hence $\pi_i(\pP) = \pP_i$ for 
$k \le i \le \ell+1$. Finally, we note that $\pP$ traverses 
the polygons traversed by $\qQ$ or $\rR$ which do not 
intersect the fiber $f^{-1}(t_m) \cap \|\fF\|$ by steps 
which move $\zeta_i \ast \eta_j$ to the first two paths 
shown in~(\ref{eq:steps}), respectively, each exactly once 
by our assumptions on $\qQ$ and $\rR$, and the 2-dimensional 
faces of $\fF$ which intersect $f^{-1}(t_m) \cap \|\fF\|$ by 
steps which move $\zeta_i \ast \eta_j$ to the third path 
shown in~(\ref{eq:steps}), each exactly once by our 
assumptions on $\pP_m$ and $\pP_{m+1}$, and that these are 
precisely the polygons of~(\ref{eq:H-diagram}).
\qed

\section{Conclusions}
\label{sec:conclusions}

The following tables summarize our results on monotone 
paths and arborescences and indicate the problems which 
remain open. 

We have no reason to doubt
that Question~\ref{que:arb-min-all} on the minimum 
number of $f$-arborescences in dimensions $d \ge 4$ and 
Question~\ref{que:lower-diam-dim3} on the maximum diameter 
of flip graphs in dimension 3 have positive answers. For 
the minimum 
diameter of flip graphs, we expect that the diameter of 
$G(P,f)$ is bounded below by the integral part of half 
the number of facets for every 3-dimensional 
polytope $P$. In particular, we expect that the 
following conjecture is true.

\begin{conjecture} \label{conj:diam-min-dim3}
The minimum diameter of $G(P,f)$, when $P$ ranges over 
all 3-dimensional polytopes with $n$ vertices and $f$ 
ranges over all generic linear 
functionals on $P$, is equal to $\left\lfloor (n+5)/4
\right\rfloor$ for every $n \ge 4$.
This can be achieved by simple polytopes for every 
even $n$.
\end{conjecture}

\medskip
\begin{center}
\begin{table}[h]
\begin{tabu}{c|c|[2pt]c|c}
\tabucline[2pt]{-}
\multicolumn{2}{c|[2pt]}{\# of arborescences} & all polytopes &  simple polytopes \\ \tabucline[2pt]{-}
\multirow{2}{*}{$d=3$}  & upper bound & Theorem \ref{thm:arb-max-all} & \multirow{2}{*}{Corollary \ref{thm:arb-simple}}   \\ \cline{2-3} 
                   & lower bound & Theorem \ref{thm:arb-min-dim3-all} & \\
                   \hline
\multirow{2}{*}{$d\geq 4$}  & upper bound & Theorem \ref{thm:arb-max-all} & Corollary \ref{thm:arb-simple} \\ \cline{2-4} 
                   & lower bound & Question \ref{que:arb-min-all} & Corollary \ref{thm:arb-simple}   \\ \tabucline[2pt]{-}
\end{tabu}
\caption{Summary for $f$-arborescences}
\end{table}
\end{center}

\begin{center}
\begin{table}[h]
\begin{tabu}{c|c|[2pt]c|c}
\tabucline[2pt]{-}
\multicolumn{2}{c|[2pt]}{\# of monotone paths} & all polytopes &  simple polytopes \\ \tabucline[2pt]{-}
\multirow{2}{*}{$d=3$}  & upper bound & Theorem \ref{thm:path-max} &  Conjecture \ref{conj:path-max-simple}, Proposition 
                          \ref{cor:path-max-simple} \\ \cline{2-4} 
                   & lower bound &  \multicolumn{2}{c}{Theorem \ref{thm:path-min-dim3}}  \\ \hline
\multirow{2}{*}{$d\geq 4$}  & upper bound & Remark \ref{thm:path-max-alldim} & open \\ \cline{2-4} 
                   & lower bound & \multicolumn{2}{c}{Proposition \ref{cor:path-min-alldim}}   \\ \tabucline[2pt]{-}
\end{tabu}
\caption{Summary for $f$-monotone paths}
\end{table}
\end{center}

\begin{center}
\begin{table}[h]
\begin{tabu}{c|c|[2pt]c|c}
\tabucline[2pt]{-}
\multicolumn{2}{c|[2pt]}{diameter of flip graph} & all polytopes & simple polytopes \\ \tabucline[2pt]{-}
\multirow{2}{*}{$d=3$}  & upper bound & Theorem \ref{thm:diam-max}, Question \ref{que:lower-diam-dim3} & open \\ \cline{2-4} 
                   & lower bound &  \multicolumn{2}{c}{Conjecture \ref{conj:diam-min-dim3}}  \\ \hline
\multirow{2}{*}{$d\geq 4$}  & upper bound & open & open \\ \cline{2-4} 
                   & lower bound & open & open \\ \tabucline[2pt]{-}
\end{tabu}
\caption{Summary for the diameter of flip graphs}
\end{table}
\end{center}

The outdegrees of the vertices of $\omega(P,f)$ play an 
important role in the proofs of Theorems~\ref{thm:main-3d} 
and~\ref{thm:main-d>3}. 
It seems a very interesting problem to characterize, or 
at least obtain significant information about, the possible 
multisets of these outdegrees when $P$ ranges over all 
polytopes of given dimension and number of vertices and 
$f$ ranges over all generic linear functionals on $P$. 
Finally, it would be interesting to address the 
questions raised in this paper for \emph{coherent} 
$f$-monotone paths as well. Their number typically 
grows much slower than the total 
number of $f$-monotone paths \cite{ADRS00}.

\medskip
\noindent {\bf Acknowledgements:} The first author warmly thanks the Department of Mathematics at UC Davis for hosting him and providing an excellent academic atmosphere during the Fall of 2019, when this paper was prepared. The second and third authors are grateful for the support received through NSF grants DMS-1818969 and HDR TRIPODS.


\bibliographystyle{plainurl}
\bibliography{extremalBib.bib}
\end{document}